\documentclass[leqno,11pt,a4paper]{amsart}
\usepackage{amsthm}
\usepackage{enumitem}
\usepackage{hyperref}
\usepackage{graphicx}
\usepackage{amssymb}
\usepackage{float}
\usepackage{xcolor}
\hypersetup{
  colorlinks,
  linkcolor={blue!50!black},
  citecolor={blue!50!black},
  urlcolor={blue!80!black}
}
\usepackage{arydshln}
\usepackage{multirow}
\usepackage[all]{xy}
\usepackage{tikz}
\usetikzlibrary{positioning}
\usepackage{tikz-cd}
\usepackage{booktabs}
\setlist[enumerate]{labelsep=*, leftmargin=1.5pc}
\setlist[enumerate]{label=\normalfont(\roman*), ref=\roman*}
\usepackage[margin=2.7cm]{geometry}
\graphicspath{{images/}}
\numberwithin{figure}{section}
\theoremstyle{plain}
\newtheorem{thm}{Theorem}[section]
\newtheorem{pro}[thm]{Proposition}
\newtheorem{lem}[thm]{Lemma}
\newtheorem{cor}[thm]{Corollary}
\newtheorem{conjecture}{Conjecture}
\theoremstyle{definition}
\newtheorem{dfn}[thm]{Definition}
\newtheorem{rem}[thm]{Remark}
\newtheorem{eg}[thm]{Example}
\newtheorem{algorithm}[thm]{Algorithm}
\DeclareMathOperator{\Cl}{Cl}

\DeclareMathOperator{\Hom}{Hom}

\DeclareMathOperator{\Pic}{Pic}

\DeclareMathOperator{\Div}{Div}

\DeclareMathOperator{\GL}{{GL}}

\newcommand{\cA}{\mathcal{A}}
\newcommand{\cO}{\mathcal{O}}
\newcommand{\QQ}{{\mathbb{Q}}}
\newcommand{\RR}{{\mathbb{R}}}
\newcommand{\PP}{{\mathbb{P}}}

\newcommand{\ZZ}{{\mathbb{Z}}}
\renewcommand{\AA}{{\mathbb{A}}}
\newcommand{\FF}{\mathbb{F}}

\newcommand{\GG}{\mathbb{G}}

\newcommand{\cX}{\mathcal{X}}

\newcommand{\cM}{\mathcal{M}}

\newcommand{\cQ}{\mathcal{Q}}

\renewcommand{\tilde}{\widetilde}

\renewcommand{\bar}{\overline}
\newcommand{\conv}[1]{\operatorname{conv}\left({#1}\right)}

\begin{document}
%-------------------------------------------------------------------------------
\author[D.~Cavey]{Daniel Cavey}
\address{School of Mathematical Sciences\\University of Nottingham\\Nottingham, NG$7$\ $2$RD\\UK}
\email{pmxdc4@nottingham.ac.uk}% Cavey
\author[T.~Prince]{Thomas Prince}
\address{Department of Mathematics\\Imperial College London\\London, SW$7$\ $2$AZ\\UK}
\email{t.prince12@imperial.ac.uk}% Prince
%-------------------------------------------------------------------------------
%\keywords{}
%\subjclass[2000]{}
%-------------------------------------------------------------------------------
\title{Del~Pezzo surfaces with a single $\frac{1}{k}(1,1)$ singularity}
\maketitle
%-------------------------------------------------------------------------------
\begin{abstract}
Inspired by the recent progress by Coates--Corti--Kasprzyk et al.\ on Mirror Symmetry for del~Pezzo surfaces, we show that for any positive integer $k$ the deformation families of del~Pezzo surfaces with a single $\frac{1}{k}(1,1)$ singularity (and no other singular points) fit into a single cascade. Additionally we construct models and toric degenerations of these surfaces embedded in toric varieties in codimension $\leq 2$. Several of these directly generalise constructions of Reid--Suzuki (in the case $k=3$). We identify a root system in the Picard lattice, and in light of the work of Gross--Hacking--Keel, comment on Mirror Symmetry for each of these surfaces. Finally we classify all del~Pezzo surfaces with certain combinations of $\frac{1}{k}(1,1)$ singularities for $k = 3,5,6$ which admit a toric degeneration.
\end{abstract}

%----------------------------------------------------------------------
\section{Introduction}
%----------------------------------------------------------------------

The smooth del~Pezzo surfaces are among the most familiar, and fundamental, objects in algebraic geometry. It has been known since the end of the 19th century that -- following the terminology of Reid--Suzuki~\cite{RS03} -- these surfaces form a \emph{cascade} (see del~Pezzo~\cite{dP1887} together with Castelnuovo's contractibility critereon~\cite{C1897}). Indeed, every smooth del~Pezzo surface is obtained from $\PP^2$ by blowing up a general collection of points, with the exception of $\PP^1\times \PP^1$ which is the contraction of an exceptional curve on $\PP^2$ blown up in two distinct points.

Fixing an algebraically closed field of characteristic zero, an analogous cascade appears when one allows the del~Pezzo surface to acquire a single $\frac{1}{k}(1,1)$ singularity (see~\S\ref{sec:del_Pezzos}). For every surface in such a cascade there is an embedding of this surface into a toric variety with codimension $\leq 2$.

\begin{thm}
\label{thm:cascade}
Given an integer $k > 3$ there are precisely $k+6$ deformation classes of del~Pezzo surfaces with a single $\frac{1}{k}(1,1)$ singularity. Of these, $k+5$ families are obtained by blowing up $\PP(1,1,k)$ in general smooth points. The remaining surface is obtained by contracting an exceptional curve on $\PP(1,1,k)$ blown up in $k+1$ smooth points. Moreover there is an embedding (not always quasismooth) of these surfaces, and a toric degeneration of each of these surfaces, into a toric variety with codimension $\leq 2$.
\end{thm}

The definitions of these families apply to any non-negative integer $k$ (by convention $k=1$ denotes the smooth case). These families account for all but one family of del~Pezzo surfaces with a single $\frac{1}{k}(1,1)$ singularity in the case $k=3$ or $k=2$, and all but three families in the case $k=1$. An example of one of the cascades is given in~\S\ref{sec:cascadesofsurfaces} for $k=5$.

It is well known that each of the ten smooth del~Pezzo surfaces is related to a certain \emph{root system}, whose roots are $-2$ classes in the orthogonal of the canonical class in the Picard group of the del~Pezzo surface. The Weyl group of this root system acts on the collection of $(-1)$-curves of the del~Pezzo surface. The list of root systems $R$ associated to the smooth del~Pezzo surfaces, listed by their degree $d$, was described by Manin~\cite{Manin86}.
\[
\begin{array}{c||c|c|c|c|c|c|c}
9-d& 2 & 3 & 4 & 5 & 6 & 7 & 8  \\
\hline
R & A_1 & A_2 \times A_1 & A_4 & D_5 & E_6 & E_7 & E_8
\end{array}
\]  
We prove the following analogous statement for the cascade of surfaces obtained from $\PP(1,1,k)$.
\begin{thm}
\label{thm:root_systems}
Let $X$ be a del~Pezzo surface obtained as the blow-up of $\PP(1,1,k)$ in $2 \leq l \leq k+4$ general smooth points. The collection of $-2$ classes in $\omega^\bot \subset \Pic(X)$ is a root system given by:
\[
\begin{array}{c||c|c|c|c|c|c}
l = (k+1)^2/k - d & 2 & \ldots & k+1 & k+2 & k+3 & k+4  \\
\hline
R & A_1 & \ldots & A_k & A_{k+1} \times A_1 & A_{k+3} & D_{k+4}
\end{array}
\]
In the case $k=3$ the blow-up of $\PP(1,1,k)$ in $l=k+5$ points is also a del~Pezzo surface, and contains an $E_8$ root system generating the Picard lattice of this surface.
\end{thm}
Note that these are all the interesting cases: the case $k=1$ is classical, and while if $k=2$ there is an additional surface given by the blow-up of $\PP(1,1,2)$ in $l=k+5$ general points the resolution of this surface is a weak smooth del~Pezzo surface, which are also well understood.

In ~\cite{CK17} a classification of all toric del~Pezzo surfaces with a $\QQ$-Gorenstein deformation to a del~Pezzo surface with only combinations of $\frac{1}{3}(1,1)$,~$\frac{1}{5}(1,1)$, and~$\frac{1}{6}(1,1)$ singularities, as listed in Theorem~\ref{thm:small_classification}, is given. Theorem~\ref{thm:cascade} tells us that there are no additional del~Pezzo surfaces; that is, all such del~Pezzo surfaces admit a toric degeneration to one of the toric varieties in~\cite{CK17}. These toric degenerations are also embedded in toric varieties with codimension $\leq 2$.

\begin{thm}[\cite{CK17}]
\label{thm:small_classification}
There are precisely twenty-six $\QQ$-Gorenstein deformation classes of surfaces with basket of singularities of the form
\[
\left\{m_1 \times \frac{1}{3}(1,1), m_2 \times \frac{1}{5}(1,1), m_3 \times \frac{1}{6}(1,1)\right\}
\]
for which either
\begin{align*}
m_1=0,m_2 > 0,m_3=0 \quad \text{or} \quad m_1 \geq 0,m_2=0,m_3 >0.
\end{align*}
All of these admit a $\QQ$-Gorenstein toric degeneration. There are precisely fourteen such families in the first case, and twelve in the second.
\end{thm}
\begin{rem}
Corti--Heuberger~\cite{CH16} identify three surfaces with a $\frac{1}{3}(1,1)$ singularity which do not admit a toric degeneration. Four log del~Pezzo surfaces whose blow-up in a smooth point does not admit a toric degeneration appear in Remark~\ref{rem:pairs_of_singularities}.
\end{rem}

Examples of these cascades of surfaces -- particularly the cascade obtained by blowing up $\PP(1,1,3)$ -- were considered by Reid--Suzuki~\cite{RS03}, where they construct equations for anti-canonically embedded del~Pezzo surfaces from a candidate Hilbert series. While this method makes contact with our approach at a number of points, our methods are essentially different: rather than a Hilbert series we start with a candidate toric variety (to which the desired surface degenerates) and then construct embeddings into (possibly quite general) toric varieties.

\emph{Laurent inversion}~\cite{CKP15} is used to construct models of del~Pezzo surfaces; this construction is briefly recalled in~\S\ref{sec:laurent_inversion}. The surfaces constructed provide examples for a number of results and conjectures in Mirror Symmetry, which are collected in~\S\ref{sec:mirror_symmetry}. One immediate consequence is to note~\cite[Conjecture~A]{A+} holds in the case of del~Pezzo surfaces with  the combinations of singularities appearing in Theorem~\ref{thm:small_classification} (see~\S\ref{sec:conjecture_A}). Furthermore, mirror models for the surfaces are given using the work of Gross--Hacking--Keel~\cite{GHK2,GHK1} and Gross--Hacking--Keel--Kontsevich~\cite{GHKK}; this uses the language of cluster algebras, and allows us to describe the complement of the anti-canonical divisor and its mirror-dual via certain quivers.

%----------------------------------------------------------------------
\section{Preliminaries on Surfaces}
%----------------------------------------------------------------------

\subsection{Del~Pezzo surfaces with cyclic quotient singularites}
\label{sec:del_Pezzos}

Let $\mu_k$ be the group generated by a primitive $k$-th root of unity. The notation $\frac{1}{k}(a,b)$ denotes the singularity obtained as the quotient of $\AA^2$ by the group $\mu_k$ acting with weights $(a,b)$. The singularity $\frac{1}{k}(1,1)$ is du~Val in the cases $k=1$,~$2$, which are smooth and ordinary double points respectively. These are the only two cases for which the singularity $\frac{1}{k}(1,1)$ is canonical.

\begin{dfn}
\label{dfn:T_and_R}
Given an arbitrary quotient singularity $\sigma = \frac{1}{R}(a,b)$, set $k = \text{gcd}(a+b,R)$, $c = (a+b)/k$ and $r=R/k$. Then $\sigma$ can be written in the form $\frac{1}{kr}(1,kc-1)$ and:
\begin{enumerate}
\item $\sigma$ is a \emph{T-singularity}~\cite{KSB88} if $r \mid k$;
\item $\sigma$ is an \emph{R-singularity}~\cite{AK14} if $k<r$.
\end{enumerate}
\end{dfn}

Definition~\ref{dfn:T_and_R} is motivated by the work of Wahl~\cite{W80} and Koll\'{a}r--Shepherd-Barron~\cite{KSB88} on the deformations of singularities. Discussion of these definitions from a toric viewpoint can be found in Akhtar--Kasprzyk~\cite{AK14}. A cyclic quotient singularity is a $T$-singularity if and only if it admits a $\QQ$-Gorenstein smoothing. Alternatively an $R$-singularity is rigid under any $\QQ$-Gorenstein deformation. 

\begin{eg}
The singularities $\frac{1}{k}(1,1)$ are $R$-singularities precisely when $k=3$ or $k\geq 5$. The singularities $\frac{1}{2}(1,1)$ and $\frac{1}{4}(1,1)$ are $T$-singularities.
\end{eg}

\noindent An algebraic surface is $\QQ$-Gorenstein if it is normal and the canonical divisor class is $\QQ$-Cartier.

\begin{dfn}
Let $X$ be a $\QQ$-Gorenstein algebraic surface; $X$ is a \emph{del~Pezzo surface} if the anti-canonical divisor class $-K_X$ is ample. Every del~Pezzo surface in this article will have quotient singularities of the form $\frac{1}{k}(a,b)$ for some integers $a$,~$b$,~$k$. The \emph{Fano index} of a del~Pezzo surface $X$ is the largest positive integer $f$ such that $K_X = f\cdot D$ for some $D \in \Cl(X)$.
\end{dfn}

Given a del~Pezzo surface $X$ with singularities of the form $\frac{1}{k}(1,1)$ the minimal resolution $\widehat{X} \rightarrow X$ contracts a unique curve $E$ (with $E^2 = -k$) for each $\frac{1}{k}(1,1)$ singularity. The anti-canonical class of $\widehat{X}$ is always big, but is only nef if all the singularities of $X$ are ordinary double points.

\begin{dfn}
A toric degeneration will refer to a flat and proper morphism $\pi\colon \cX \rightarrow S$ of normal schemes for which $S$ has a distinguished point $0 \in S$ such that the fibre $X_0$ is a normal toric variety. A toric degeneration $\cX \rightarrow S$ is \emph{$\QQ$-Gorenstein} if the relative anti-canonical divisor class $-K_{\cX/S}$ is $\QQ$-Cartier and relatively ample.
\end{dfn}

\subsection{Quasismooth surfaces}
\label{sec:quasismooth}

Following Iano-Fletcher~\cite{IF00}, let us recall the notion of a quasismooth complete intersection in weighted projective space $w\PP=\PP(a_{0},\ldots,a_{n})$. Let $X\subset w\PP$ be a closed subvariety, and let $\rho\colon \AA^{n+1} \backslash \{ 0 \} \rightarrow w\PP$ denote the canonical projection. The \emph{punctured affine cone} is given by $C^{\circ}_{X} = \rho^{-1}(X)$, and the \emph{affine cone} $C_{X}$ over $X$ is the completion of $C_{X}^{\circ}$ in $\AA^{n+1}$. Note that the usual action of the group $K^*$ on $w\PP$ can be restricted to $C^{\circ}_{X}$, and $X = C^{\circ}_{X} / K^*$ (here $K$ denotes our fixed algebraically closed field of characteristic zero). $X \subset w\PP$ is \emph{quasismooth of dimension $m$} if its affine cone $C_{X}$ is smooth of dimension $m+1$ outside its vertex $\underline{0}$. When $X \subset w\PP$ is quasismooth the singularities of $X$ are due to the $K^*$-action and hence are cyclic quotient singularities.
\medskip

\begin{thm}[{\cite[Theorem~8.1]{IF00}}]
\label{thm:quasismooth_codim_1}
The general hypersurface $X_{d}\subset\PP(a_{0},\ldots,a_{n})$, where $n \geq 1$, is quasismooth if and only if one of the following holds:
\begin{enumerate}
\item there exists a coordinate $x_{i}$ of $\PP(a_{0},\ldots,a_{n})$ for some $i$ of weight $d$; or
\item for every non-empty subset $I = \{ i_{0}, \ldots, i_{k-1} \} \subset \{ 0,\ldots,n \}$ either:
\begin{enumerate}
\item[(a)] there exists a monomial $x_{i_{0}}^{m_{0}}\cdots x_{i_{k-1}}^{m_{k-1}}$ of degree $d$; or
\item[(b)] for $\mu = 1,\ldots,k$ there exist monomials $x_{i_{0}}^{m_{0,\mu}}\cdots x_{i_{k-1}}^{m_{k-1,\mu}}x_{e_{\mu}}$ of degree $d$, where each of the $e_{\mu}$ are distinct.
\end{enumerate}
\end{enumerate} 
\end{thm} 

\begin{thm}[{\cite[Theorem~8.7]{IF00}}]
\label{thm:quasismooth_codim_2}
Consider a codimension two weighted complete intersection $X_{d_{1},d_{2}}\subset\PP(a_{0},\ldots,a_{n})$, where $n \geq 2$, which is not the intersection of a linear cone with another hypersurface. $X_{d_{1},d_{2}}$ is quasismooth if and only if for each non-empty subset $I= \{ i_{0},\ldots, i_{k-1} \} \subset \{ 0,\ldots,n \}$ one of the following holds:
\begin{enumerate}
\item there exist monomials $x_{i_{0}}^{m_{1,0}}\cdots x_{i_{k-1}}^{m_{1,k-1}}$ and $x_{i_{0}}^{m_{2,0}}\cdots x_{i_{k-1}}^{m_{2,k-1}}$ of degree $d_1$ and $d_2$, respectively;
\item there exists a monomial $x_{i_{0}}^{m_{1,0}}\cdots x_{i_{k-1}}^{m_{1,k-1}}$ of degree $d_{1}$ and for $\mu = 1,\ldots,k-1$ there exist monomials $x_{i_{0}}^{m_{2,0}}\cdots x_{i_{k-1}}^{m_{2,k-1}} x_{e_{\mu}}$ of degree $d_{2}$ where the $\{ e_{\mu} \}$ are all distinct;
\item there exists a monomial $x_{i_{0}}^{m_{2,0}}\cdots x_{i_{k-1}}^{m_{2,k-1}}$ of degree $d_{2}$ and for $\mu = 1,\ldots,k-1$ there exist monomials $x_{i_{0}}^{m_{1,0}}\cdots x_{i_{k-1}}^{m_{1,k-1}} x_{e_{\mu}}$ of degree $d_{1}$ where the $\{ e_{\mu} \}$ are all distinct;
\item for $\mu = 1,\ldots, k$ there exist monomials $x_{i_{0}}^{m_{1,0}}\cdots x_{i_{k-1}}^{m_{1,k-1}}x_{e_{\mu}^{1}}$ and $x_{i_{0}}^{m_{2,0}}\cdots x_{i_{k-1}}^{m_{2,k-1}} x_{e_{\mu}^{2}}$ of degrees $d_1$ and $d_2$, respectively, such that $\{ e_{\mu}^{1} \}$ are all distinct, $\{ e_{\mu}^{2} \}$ are all distinct and $\{ e_{\mu}^{1}, e_{\mu}^{2} \}$ contains at least $k+1$ distinct elements.
%For $\mu = 1,\ldots, k$ there exist monomials
%$x_{i_{0}}^{m_{1,0}}\cdots x_{i_{k-1}}^{m_{1,k-1}}x_{e_{\mu}^{1}}$ of degree $d_{1}$ and $x_{i_{0}}^{m_{2,0}}\cdots x_{i_{k-1}}^{m_{2,k-1}} x_{e_{\mu}^{2}}$ of degree $d_{2}$, such that $\{ e_{\mu}^{1} \}$ are all distinct, $\{ e_{\mu}^{2} \}$ are all distinct and $\{ e_{\mu}^{1}, e_{\mu}^{2} \}$ contains at least $k+1$ distinct elements.
\end{enumerate}
\end{thm}

\subsection{Cascades of surfaces}
\label{sec:cascadesofsurfaces}

For each integer $k=3$ or~$k>4$ we study a cascade of surfaces obtained from the weighted projective space $\PP(1,1,k)$ by blowing up general points and contracting exceptional curves. In fact, as mentioned in the introduction, the cascades are particularly simple: all but one surface in each cascade is obtained from $\PP(1,1,k)$ via blow-up in $k+4$ general smooth points.

\begin{eg}
The cascade of del~Pezzo surfaces with a single $\frac{1}{5}(1,1)$ singularity is:
\begin{figure}[H]
\centering
\begin{tikzpicture}
[-,auto,node distance=1.5cm, thick,main node/.style={circle,draw,font=\sffamily \Large\bfseries}]
  \node[text=black] (1) {$\mathbb{P}(1,1,5)$};
  \node[text=black] (2) [right=0.5cm of 1] {$X_{5}^{(1)}$};
  \node[text=black] (3) [right of=2] {$X_{5}^{(2)}$};
  \node[text=black] (4) [right of=3] {$X_{5}^{(3)}$};
  \node[text=black] (5) [right of=4] {$X_{5}^{(4)}$};
  \node[text=black] (6) [right of=5] {$X_{5}^{(5)}$};
  \node[text=black] (7) [right of=6] {$X_{5}^{(6)}$};
  \node[text=black] (8) [right of=7] {$X_{5}^{(7)}$};
  \node[text=black] (9) [right of=8] {$X_{5}^{(8)}$};
  \node[text=black] (10) [right of=9] {$X_{5}^{(9)}$};
  \node[text=black] (11) [below of=7] {$B_{5}^{(5)}$};
 
  \path[every node/.style={font=\sffamily}]
    (2) edge[->] (1)
    (3) edge[->] (2)
    (4) edge[->] (3)
    (5) edge[->] (4)
    (6) edge[->] (5)
    (7) edge[->] (6)
    (8) edge[->] (7)
    (9) edge[->] (8)
    (10) edge[->] (9)
    (7) edge[->] (11);

\end{tikzpicture}
\end{figure}
\noindent Properties of these surfaces are given in the following table:

\begin{center}
\begin{tabular}{c | c | c}
Surface & Fano Index & Is toric \\ \hline \hline
$\PP(1,1,5)$ & 7 & Yes \\ \hline
$X_{5}^{(i)}$, for $i \in \{1,2\}$ & 1 & Yes \\ \hline
$X_{5}^{(i)}$, for $i \in \{ 3,4,\ldots,9 \}$ & 1 & No \\ \hline
$B_{5}^{(5)}$ & 2 & No \\ \hline
\end{tabular}
\end{center}
\bigskip

\noindent These properties generalise to any cascade appearing in Theorem~\ref{thm:cascade} in the obvious way. The above table also illustrates how our work overlaps with the classifications of del~Pezzo surfaces with Fano index $> 1$ by Alexeev--Nikulin~\cite{AN06} and Fujita--Yasutake~\cite{FY17}, and the classification by Dais~\cite{DToricdPs} of toric del~Pezzo surfaces with exactly one singular point.
\end{eg}

\begin{dfn}
\label{dfn:cascade_1}
For a given $k \in \ZZ_{> 0}$ let $X_k := \PP(1,1,k)$ and let $X^{(l)}_k$ denote the blow-up of $\PP(1,1,k)$ in $l$ general points. Assume that 
\[
l < \frac{(k+2)^2}{k}.
\]
\end{dfn}

\begin{rem}
The degree of $\PP(1,1,k)$ is $(k+2)^2/k$ and thus the bound on $l$ in Definition~\ref{dfn:cascade_1} ensures that $X^{(l)}_k$ is a del~Pezzo surface.
\end{rem}

\noindent The cascade consists of the surfaces $X^{(l)}_k$ for a fixed value of $k$ and all possible values of $l$, along with an additional surface obtained by contracting a curve on $X^{(k+1)}_k$.

\begin{dfn}
\label{dfn:cascade_2}
Fix a positive integer $k$ and $k+1$ points $\{p_i : 1 \leq i \leq k+1 \}$ on $\PP(1,1,k)$. There is a unique curve $C$ in the linear system $\cO(k)$ passing through these $k+1$ points. Blow-up all the points $p_i$ and let $C'$ be the strict transform of the curve $C$. Let $B^{(k)}_k$ denote the surface obtained by contacting $C'$.
\end{dfn}
This is the obvious generalisation of the construction of $B^{(1)}_1 \cong \PP^1 \times \PP^1$ from $\PP^2$. In our constructions of low codimension models for the surfaces $X^{(l)}_{k}$,~$B^{(k)}_k$ we make use of alternate descriptions of $X^{(k+2)}_k$, $X^{(k+3)}_k$, and $X^{(k+4)}_{k}$ depending on the parity of $k$.
\begin{dfn}
\label{dfn:S_k_plus_2}
Fix a positive integer $k$ and $(k+2)$ points  $\{p_i : 1 \leq i \leq k+2 \}$ on the diagonal $\Delta \subset \PP^1 \times \PP^1$. Let $S_k$ denote the surface obtained by blowing up the points $p_i$. Letting $\Delta$ also denote the strict transform of the diagonal, it follows immediately that $\Delta^2 = -k$.
\end{dfn}

\begin{lem}
\label{lem:alternate_blowup}
The surface $S_k$ is a minimal resolution of $X^{(k+2)}_k$. The resolution contracts the strict transform of the diagonal in $\PP^1\times \PP^1$.
\end{lem}
\begin{proof}
Let $\pi_j$,~$j=1$,~$2$ denote the $j$th projection $\pi_j \colon \PP^1\times\PP^1 \rightarrow \PP^1$ and let $E_i$ denote the strict transform of the fibre $\pi^{-1}_1(\pi_1(p_i))$. Each morphism $\pi_j$ induces a morphism $S_k \rightarrow \PP^1$ with $k+2$ reducible fibres. Each of these fibres contains precisely one of the curves $E_i$. Thus, by contracting all the curves $E_i$, obtain a surface $\tilde{S}_k$ together with a morphism $\tilde{S}_k \rightarrow \PP^1$ such that all its fibres are isomorphic to $\PP^1$. That is, $\tilde{S}_k$ is isomorphic to the Hirzebruch surface $\FF_k$. Consider the following commuting diagram:
\[ \begin{tikzcd}[row sep=large, column sep = large]
S_{k}  \arrow{r} \arrow[swap]{d} &X_{k}^{(k+2)} \arrow{d} \\
\tilde{S}_k \arrow{r} & \PP(1,1,k)
\end{tikzcd}
\]
Thus $S_{k} \rightarrow X_{k}^{k+2}$ is a minimal resolution.
\end{proof}

\begin{dfn}
\label{dfn:S_k_plus_3}
Fix a positive integer $k$ and $k+4$ points $\{p_i : 1 \leq i \leq k+4 \}$ on a conic in $\PP^2$. Let $S'_k$ denote the surface obtained by blowing up the points $p_i$. If $C$ denotes the strict transform of the conic, it follows immediately that $C^2 = -k$.
\end{dfn}

\begin{lem}
\label{lem:alternate_blowup_2}
The surface $S'_k$ is a minimal resolution of $X^{(k+3)}_k$. The resolution contracts the strict transform of the conic in $\PP^2$ used to define $S'_k$.
\end{lem}
\begin{proof}
Let $C$ be a conic in $\PP^2$ and fix $k+4$ points $\{p_i : 1 \leq i \leq k+4 \}$ on $C$. Consider the surface obtained by blowing up only $p_{k+4}$ and the strict transform of $C$. The blow-up is isomorphic to the first Hirzebruch surface $\FF_1$. Let $\pi \colon \FF_1 \rightarrow \PP^1$ be its projection to $\PP^1$. Blow-up the points $p_i$,~$ 1 \leq i \leq k+3$ and contract the strict transforms of the fibres $\pi^{-1}(\pi(p_i))$ of $\pi$. In this way obtain a ruled surface with a unique $-k$ curve, i.e. the surface $\FF_k$, the minimal resolution of $\PP(1,1,k)$. By a similar argument to Lemma~\ref{lem:alternate_blowup}, $S_{k}' \rightarrow X_{k}^{k+3}$ is a minimal resolution.
\end{proof}

\begin{rem}
Consider the anti-canonical degree 
\[
\Big( -K_{X^{(l)}_k} \Big)^2 = k-l+4+\frac{4}{k}.
\]
In the case $k=1$ of the smooth del~Pezzo surfaces the most interesting surfaces are those with degree $\leq 3$. However, once $k > 4$ the interesting cases from the end of the cascade are lost, even though the cascades grow in length: for large values of $k$ there are no surfaces with geometry as rich as the cubic surface or the lower degree del~Pezzo surfaces. The cases $k=2$,~$4$ are closely related to the smooth del~Pezzo surfaces (via $\QQ$-Gorenstein smoothing) and the case $k=3$ is considered in detail in~\cite{CH16}.
\end{rem}

\subsection{Hilbert Series}
\label{sec:Hilbert_Series}
%We will notice that this is not always as efficient at finding models in low codimension as the Laurent inversion method. 
We study the Hilbert series of the blow-up of $\PP(1,1,k)$ in $l \in \{ k+2, k+3, k+4 \}$ general points. Following~\cite{RS03}, consider the Hilbert series of $\PP(1,1,k)$ polarised by the anti-canonical divisor $-K_{\PP(1,1,k)} = \cO(k+2)$. This can be calculated by taking the Hilbert series of $\PP(1,1,k)$ polarised by $\cO(1)$ given by 
\[ \frac{1}{(1-s)^{2}(1-s^{k})}, \] multiplying through by $(1-s^{k+2})^{2}(1-s^{k(k+2)})$, truncating to the polynomial consisting only of terms divisible by $t^{k+2}$, and making the substitution $s^{k+2}=t$. The calculation splits into two cases:
\begin{enumerate}
\item $k$ is even. In this case, letting $k=2m$, obtain
\begin{align*}
&H_{\PP(1,1,k)} = \frac{P_{\PP(1,1,k)}(t)}{(1-t)^{2}(1-t^{k})}, \\
\text{ where }P_{\PP(1,1,k)}(t)=1 + &\sum\limits_{i=1}^{m-1} (k+4) t^{i} + (k+5)t^{m} + (k+5)t^{m+1} + \sum\limits_{i=m+2 }^{k} (k+4)t^{i} +t^{k+1} .
\end{align*}
\medskip
\item $k$ is odd. In this case, letting $k=2m-1$, obtain
\begin{align*}
&H_{\PP(1,1,k)} = \frac{P_{\PP(1,1,k)}(t)}{(1-t)^{2}(1-t^{k})}, \\
\hspace*{-2.4cm} \text{ where }  P_{\PP(1,1,k)}(t) = 1 + &\sum\limits_{i=1}^{m-1} (4+k) t^{i} + (k+6)t^{m} + \sum\limits_{i=m+1 }^{k} (k+4)t^{i} +t^{k+1} .
\end{align*}
\medskip
\end{enumerate}

\noindent A smooth blow-up has a Hilbert contribution 
\[
Q = -\frac{t}{(1-t)^{3}} = -\frac{t(1-t^{k})}{(1-t)^{3}(1-t^{k})} = -\frac{t+t^{2}+t^{3}+t^{4}+ \ldots + t^{k}}{(1-t)^{2}(1-t^{k})},
\]
and hence the Hilbert series of $X^{(l)}_k$ is 
\[
 H_{\PP(1,1,k)} + l \times Q,
\]
%\begin{align*} 
%H_{X_{k}^{(l)}} &= H_{\mathbb{P}(1,1,k)} + l \times Q  \\ &= \frac{1 + \sum\limits_{i=1}^{\frac{k-2}{2}} (4+k-l) t^{i} + (k+5-l)t^{\frac{k}{2}} + (k+5-l)t^{\frac{k+2}{2}} + \sum\limits_{i=\frac{k+4}{2} }^{k} (k+4-l)t^{i} +t^{k+1} }{(1-t)^{2}(1-t^{k})}. 
%\end{align*}
\noindent for all values of $k \in \ZZ_{\geq 1}$.
Calculating the Hilbert series for $l \in \{ k+2, k+3, k+4 \}$ suggests a low codimension model for the surface $X_{k}^{(l)}$ in each case. When these models occur in codimension $\leq 2$ they coincide with the models obtained by Laurent inversion in~\S\ref{sec:low_codim_models}; when these models occur in codimension three or four we present a different model in~\S\ref{sec:low_codim_models} which is compared with the model suggested by the Hilbert series. First consider the case $k=2m$ for some $m \in \ZZ_{\geq 1}$:
\medskip

\begin{center}
{\renewcommand{\arraystretch}{2}
\begin{tabular}{c | c | c}
$l$ & Hilbert Series & Suggested Model \\ \hline \hline
%$k+4$ & $\frac{1-t^{k}-t^{k+2}+t^{2k+2}}{(1-t)^{2}(1-t^{\frac{k}{2}})(1-t^{\frac{k+2}{2}})(1-t^{k})}$ & $X_{k,k+2} \subset \mathbb{P}(1,1,\frac{k}{2},\frac{k+2}{2},k)$ \\ \hline
% UNFACTORIZED $k+4$ & {\small$\frac{1-t^{k}-t^{k+2}+t^{2k+2}}{(1-t)^{2}(1-t^{m})(1-t^{m+1})(1-t^{k})}$} & $X_{k+2} \subset \PP(1,1,m,m+1)$ \\ \hline
$k+4$ & {\small$\frac{1-t^{k+2}}{(1-t)^{2}(1-t^{m})(1-t^{m+1})}$} & $X_{k+2} \subset \PP(1,1,m,m+1)$ \\ \hline
$k+3$ & {\small $\frac{1-t^{m+2}}{(1-t)^{3}(1-t^{m})}$} & $X_{m+2} \subset \PP(1,1,1,m)$ \\ \hline
%$k+2$ & {\small $\frac{1-t^{2}-t^{m+1}-t^{k}+t^{m+3}+t^{k+2} + t^{3m+1} -t^{3m+3}}{(1-t)^{4}(1-t^{m})(1-t^{k})}$} & $X_{2,m+1} \subset \PP(1,1,1,1,m)$ 
$k+2$ & {\small $\frac{(1-t^2)(1-t^{m+1})}{(1-t)^{4}(1-t^{m})}$} & $X_{2,m+1} \subset \PP(1,1,1,1,m)$ 
\end{tabular}}
\end{center}
\medskip
Consider the case $k=2m-1$ for some $m \in \ZZ_{\geq 1}$:
\medskip

%Looking back at the odd case and letting $k=2m+1$, we have the following table of Hilbert series.
%\[ H_{X_{k}^{(l)}} = \frac{1 + \sum\limits_{i=1}^{\frac{k-1}{2}} (4+k-l) t^{i} + (k+6-l)t^{\frac{k+1}{2}} + \sum\limits_{i=\frac{k+3}{2} }^{k} (k+4-l)t^{i} +t^{k+1} }{(1-t)^{2}(1-t^{k})}. \]
%Therefore:
%\medskip
%\begin{tabular}{c | c | c}
%Value of $l$ & Hilbert Series & Suggested Model \\ \hline \hline
%$k+4$ & $\frac{1-2t^{k+1}+t^{2k+2}}{(1-t)^{2}(1-t^{\frac{k+1}{2}})^{2}(1-t^{k})}$ & $X_{k+1,k+1} \subset \mathbb{P}(1,1,\frac{k+1}{2},\frac{k+1}{2},k)$ \\ \hline
%$k+3$ & $\frac{1-2t^{\frac{k+3}{2}}-3t^{k+1}+3t^{k+2} + 2t^{\frac{3k+3}{2}} -t^{2k+3}}{(1-t)^{3}(1-t^{\frac{k+1}{2}})^{2}(1-t^{k})}$ & $\text{Pf}_{5,5} \subset %\mathbb{P}(1,1,1,\frac{k+1}{2},\frac{k+1}{2},k)$ \\ \hline
%$k+2$ & $\frac{1-t^{2}-4t^{\frac{k+3}{2}}+4t^{\frac{k+5}{2}}-4t^{k+1}+8t^{k+2} -4t^{k+3} + 4t^{\frac{3k+3}{2}} -4t^{\frac{3k+5}{2}} -t^{2k+2} + t^{2k+4} %}{(1-t)^{4}(1-t^{\frac{k+1}{2}})^{2}(1-t^{k})}$ & codim 4 
%\end{tabular}
\begin{center}
{\renewcommand{\arraystretch}{2}
\begin{tabular}{c | c | c}
$l$ & Hilbert Series & Suggested Model \\ \hline \hline
$k+4$ & {\small $\frac{(1-t^{k+1})^2}{(1-t)^{2}(1-t^{m})^{2}(1-t^{k})}$} & $X_{k+1,k+1} \subset \PP(1,1,m,m,k)$ \\ \hline
$k+3$ & {\small $\frac{1-2t^{m+1}-3t^{k+1}+3t^{k+2} + 2t^{3m} -t^{2k+3}}{(1-t)^{3}(1-t^{m})^{2}(1-t^{k})}$} & $\text{Pf}_{5,5} \subset \PP(1,1,1,m,m,k)$ \\ \hline
$k+2$ & {\small $\frac{1-t^{2}-4t^{m+1}+4t^{m+2}-4t^{k+1}+8t^{k+2} -4t^{k+3} + 4t^{3m} -4t^{m+2} -t^{2k+2} + t^{2k+4} }{(1-t)^{4}(1-t^{m})^{2}(1-t^{k})}$} & codim 4 
\end{tabular}}
\end{center}
\medskip
%THIS LAST ONE LOOKS LIKE CODIM 4 (9 RELATIONS, 16 SYZYGIES) BUT I CANT SEE WHY THERE IS A SYZYGY OF WEIGHT $\frac{k+5}{2}$. I WILL DOUBLE CHECK

Note that the models for odd values of $k$ generally appear in higher codimension. For odd values of $k$ the codimension appearing in the unprojection cascade directly generalises case $k=1$ (that is, of the original ten del~Pezzo surfaces). The proto-typical case for even values of $k$ is the case $k=2$, and each of the surfaces $X^{(l)}_2$ admits a smoothing to the surface $X^{(l+1)}_1$. Thus $X^{(4)}_2$,~$X^{(5)}_2$, and $X^{(6)}_2$ admit a smoothing to the del~Pezzo surfaces of degrees $4$,~$3$ and $2$ respectively, which are all known to have models of codimension $\leq 2$ in weighted projective spaces.

\begin{rem}
\label{rem:small_k}
In the cases $k=2$ and $k=4$ observe that all the constructions tabulated above are well known models of del~Pezzo surfaces. This is expected, since the $\frac{1}{k}(1,1)$ singularities are $T$-singularities precisely in these two cases.
\end{rem}
%----------------------------------------------------------------------
\section{Laurent Inversion}
\label{sec:laurent_inversion}
%----------------------------------------------------------------------

In this section we recall the method of \emph{Laurent inversion}~\cite{CKP15}, which is used to construct models for the surfaces in these cascades. We freely use definitions and basic results in toric geometry: see the books by Cox--Little--Schenck and Fulton~\cite{Cox--Little--Schenck,Fulton93} for more details on this subject.

Broadly speaking Laurent inversion takes a polytope $P$ together with a certain decoration of $P$ (called a \emph{scaffolding}) as input and returns a torus invariant embedding of the toric variety associated to $P$.

\subsection{Scaffolding}

Let $N$ be a lattice and recall that an integral polytope $P \subset N_{\QQ} := N \otimes \QQ$ is said to be \emph{Fano} if it has primitive vertices, contains the origin in its interior and is full dimensional in $N$. A scaffolding of a Fano polytope $P$ is a presentation of $P$ as the convex hull of a collection of polyhedra of sections of nef divisors on a (fixed) toric variety. We restrict our interest to the case of $N$ being a rank two lattice.

\begin{dfn}[\cite{CKP15}]
\label{dfn:scaffolding}
Fix the following data:
\begin{enumerate}
\item a lattice $N \cong \ZZ^2$ with a decomposition $N = \bar{N} \oplus N_{U}$. Denote the dual lattice by $M:=\Hom (N,\ZZ)$ and the dual decomposition  $M=\bar{M} \oplus M_{U}$;
\item a Fano polygon $P \subset N_{\QQ}$;
\item a projective toric variety $Z$, known as the \emph{shape}, given by a fan in $\bar{M}$ whose rays span $\bar{M}$.
\end{enumerate}
A \emph{scaffolding} of $P$ is a set of pairs $(D,\chi)$, known as \emph{struts}, where $D$ is a nef divisor on $Z$ and $\chi$ is an element of $N_{U}$ such that
\[
P = \text{conv}\Big(P_{D} + \chi : (D,\chi) \in S \Big),
\]
where $P_D$ is the polyhedron of sections of the torus invariant divisor $D$.
\end{dfn}

\begin{rem}
\label{rem:scaff_assumptions}
Although not required by the definition, impose two additional assumptions to simplify the Laurent inversion algorithm:
\begin{enumerate}
\item every vertex of $P$ is met by precisely one strut;
\item there is a basis $\{e_i : 1 \leq i \leq \dim N_U \}$ of $N_U$ such that the pair $(\cO,e_i) \in S$ for all values of $i$. We say, following~\cite{CKP15}, that these struts correspond to `uneliminated variables'.
\end{enumerate}
\end{rem}

\begin{eg}
First fix the data (i)--(iii) appearing in Definition~\ref{dfn:scaffolding}. Let $N$ be a rank two lattice with $N_{U} = \{ 0 \}$. Thus $M \cong \ZZ^2$ and $M_{U} = \{ 0\}$. Consider the Fano polygon $P$ with vertices $(0,1), (1,0), (1,-1), (0,-1), (-1,0), (-1,1)$, and choose $Z = \PP^2$. The fan $\Sigma_Z$ corresponding to $Z$ is:
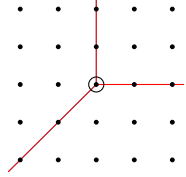
\begin{center}
\begin{tikzpicture}[baseline=-0.65ex,scale=0.5, transform shape]
\begin{scope}
\clip (-2.3,-2.3) rectangle (2.3cm,2.3cm); % Clips the picture...
\filldraw[fill=cyan, draw=red] (0,0) -- (3,0); % Puts the shaded rectangle
\filldraw[fill=cyan, draw=red] (0,0) -- (0,3);
\filldraw[fill=cyan, draw=red] (0,0) -- (-3,-3);
\foreach \x in {-7,-6,...,7}{                           % Two indices running over each
    \foreach \y in {-7,-6,...,7}{                       % node on the grid we have drawn 
    \node[draw,shape = circle,inner sep=1pt,fill] at (\x,\y) {}; % Places a dot at those points
    }
}
 \node[draw,shape = circle,inner sep=4pt] at (0,0) {}; % Places a dot at those points
\end{scope}
\end{tikzpicture}
\end{center}
%We know that the nef divisors on $Z$ are given exactly by convex piecewise functions on the rays of $\Sigma_{Z}$.
Let $\Sigma_{Z}(1) = \{ \sigma_{1}, \sigma_{2}, \sigma_{3} \}$ and denote the generator of the ray $\sigma_{i}$ by $\rho_{i}$. Define a piecewise-linear function
\[
\phi_{i}(\rho_{j}) := 
\begin{cases} 
1, & \text{ if } i=j \\
0, & \text{ otherwise. }
\end{cases}
\]
Denote the divisor corresponding to $\phi_i$ by $D_i$. Consider the scaffold given by the three struts $(D_{1},0)$, $(D_{2},0)$ and $(D_{3},0)$. Computing the polyhedra $P_{D_{i}}$ obtain
\[ P_{D_{1}} = \Bigg\{ (x,y) \in N_{\RR} : \begin{array}{c} \big\langle (x,y) , (1,0) \big\rangle \geq -1 \\ \big\langle (x,y) , (0,1) \big\rangle \geq 0 \\ \big\langle (x,y) , (-1,-1) \big\rangle \geq 0 \end{array} \Bigg\} = \Bigg\{ (x,y) \in N_{\mathbb{R}} : \begin{array}{c} x \geq -1 \\ y \geq 0 \\ x+y \leq 0 \end{array} \Bigg\}, \]
\[ P_{D_{2}} = \Bigg\{ (x,y) \in N_{\RR} : \begin{array}{c} x \geq 0 \\ y \geq -1 \\ x+y \leq 0 \end{array} \Bigg\}, \]
\[ P_{D_{3}} = \Bigg\{ (x,y) \in N_{\RR} : \begin{array}{c} x \geq 0 \\ y \geq 0 \\ x+y \leq 1 \end{array} \Bigg\}, \]
which is illustrated below.

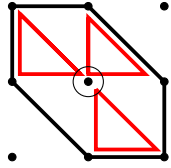
\begin{center}
\begin{tikzpicture}[transform shape]
\begin{scope}
\clip (-1.3,-1.3) rectangle (1.3cm,1.3cm); % Clips the picture...
\draw[line width=0.5mm,black] (-1,1) -- (0,1) -- (1,0) -- (1,-1) -- (0,-1) -- (-1,0) -- (-1,1); % Puts the shaded rectangle
\draw[line width=0.5mm, red] (-0.9,0.9) -- (-0.1,0.1) -- (-0.9,0.1) -- (-0.9,0.9) -- (-0.1,0.1); 
\draw[line width=0.5mm, red] (0.0,0.1) -- (0.0,0.85) -- (0.75,0.1) -- (0.0,0.1) -- (0.0,0.85);
\draw[line width=0.5mm, red] (0.1,-0.1) -- (0.9,-0.9) -- (0.1,-0.9) -- (0.1,-0.1); 
\foreach \x in {-7,-6,...,7}{                           % Two indices running over each
    \foreach \y in {-7,-6,...,7}{                       % node on the grid we have drawn 
    \node[draw,shape = circle,inner sep=1pt,fill] at (\x,\y) {}; % Places a dot at those points
    }
}
 \node[draw,shape = circle,inner sep=4pt] at (0,0) {}; % Places a dot at those points
\end{scope}
\end{tikzpicture}
\end{center}

\end{eg}

\begin{rem}
\label{rem:types_of_scaffolding}
With the exception of the scaffolding appearing in Figure~\ref{fig:anti_canonical_cod3} we will only use three types of scaffolding:
\begin{enumerate}
\item $N = \ZZ^2$,~$N_U=\ZZ$, $Z=\PP^1$;
\item $N = \ZZ^2$,~$N_U=\{0\}$, $Z=\PP^1\times\PP^1$;
\item $N = \ZZ^2$,~$N_U=\{0\}$, $Z=\PP^2$.
\end{enumerate}

\noindent Examples of these three types of scaffolding can be found in~\S\ref{sec:upto_k_plus_2},~\S\ref{sec:k_plus_2}, and~\S\ref{sec:k_plus_3} respectively.
\end{rem}

\subsection{Laurent Inversion}

Laurent inversion~\cite{CKP15} is an algorithm to pass from a scaffolding $S$ of a Fano polytope $P$ to an embedding of the corresponding Fano toric variety $X_{P}$ in an ambient toric variety $Y_S$. The form of the algorithm~\ref{alg:laurent_inversion} presented applies to a scaffolding with shape $Z$ isomorphic to a product of projective spaces; note this is true in all three cases enumerated in Remark~\ref{rem:types_of_scaffolding}.

%%%%

\begin{algorithm}[\cite{CKP15}]
	\label{alg:laurent_inversion}
	Let $S$ be a scaffolding of a Fano polytope $P$ with shape $Z$. Let $u = \dim N_U$ and let $r = |S| - u$, so that $S$ contains $r$ struts that do not correspond to uneliminated variables and $u$ struts that do correspond to uneliminated variables (see Remark~\ref{rem:scaff_assumptions}). Let $R$ be the sum of $|S|$ and the number $z$ of rays of $Z$. We determine an $ r \times R$ matrix $\cM$, which will be the weight matrix for our toric variety $Y$, as follows. Let $m_{i,j}$ denote the $(i,j)$ entry of $\cM$. Fix an identification of the rows of $\cM$ with the $r$ elements $(D_i,\chi_i)$ of $S$ which do not correspond to uneliminated variables, and an ordering $\Delta_1,\ldots,\Delta_z$ of the toric divisors in $Z$. Let $e_1,\ldots,e_u$ be the basis of $N_U$ given by Remark~\ref{rem:scaff_assumptions}.
	\begin{enumerate}
		\item For $1 \leq j \leq r$ and any $i$, let $m_{i,j} = \delta_{i,j}$.
		\item For $1 \leq j \leq u$ and any $i$, let $m_{i,r+j}$ be determined by the expansion 
		\[
		\chi_i = \sum_{j=1}^u{m_{i,r+j}e_j}.
		\]
		\item For $1 \leq j \leq z$, let $m_{i,|S|+j}$ be determined by the expansion
		\[
		D_i = \sum_{j=1}^z{m_{i,|S|+j}}\Delta_j.
		\]
	\end{enumerate}
	The weight matrix $\cM$ alone does not determine a unique toric variety -- a stability condition $\omega$ also needs to be specified. Unless otherwise stated, assume $\omega$ to be the sum of the first $|S|$ columns in $\cM$. Let $Y_\omega$ denote the toric variety determined once this choice has been made.
\end{algorithm}

\noindent This algorithm determines a toric variety $Y_S$. After choosing bases of $N_U$ and $\Div_{T_{\bar{M}}}(Z)$ the fan determined by $Y_S$ is contained in $(N_U \oplus \Div_{T_{\bar{M}}}(Z)) \otimes \QQ$.

%Here $\LL$ is defined in the short exact sequence for the fan $\Sigma_{X_{P}} \subset N_{\QQ}$ given by
%\[ \xymatrix{ 0 \ar[rr] & & \LL \ar^{D^\star}[rr] & & \ZZ^{\Sigma(1)} \ar^R[rr] & & N  \ar[rr] & & 0. } \]

\begin{thm}[\cite{CKP15}]
\label{thm:embedding}
Given a scaffolding $S$ of a Fano polytope $P$ the GIT data $(\cM,\omega)$ define a toric variety $Y_S$ with $\Cl(Y_S) \cong \ZZ^r$. Furthermore, there is a canonical embedding $X_P \hookrightarrow Y_S$. If $Z$ is isomorphic to a product of $k$ projective spaces, $X_P$ is the intersection of $k$ divisors, each of which is defined by a single equation in Cox co-ordinates, on $Y_S$, and 
\[
\omega = -K_{X} - \sum\limits_{i} L_{i},
\]
where the linear systems $L_i$ define $X_P$.
\end{thm}
Of course, if $Y_S$ is smooth these define a complete intersection. In general this needs to verified on a case-by-case basis. There are many ways of embedding a toric variety into another toric variety, but Theorem~\ref{thm:embedding} allows us to unify a large number of classical constructions of Fano varieties into a simple format. For example, given a Fano polygon $P$ there is standard choice of scaffolding, obtained by taking $Z$ to be the toric variety associated to the normal fan of $P$. This recovers the anti-canonical embedding of $X_P$.

\begin{dfn}[\cite{CKP15}]
Fix a Fano polygon $P$ and let $Z$ be the minimal resolution of the toric variety determined by the normal fan of $P$. The \emph{anti-canonical scaffolding} of $P$ is the scaffolding $S$ with shape $Z$ consisting of the single nef divisor $D$ on $Z$ such that the polyhedron of sections of $D$ is equal to $P$.
\end{dfn}

The Laurent inversion algorithm applied to the anti-canonical scaffold determines an embedding of $X_P$ into the weighted projective space $\PP(1,a_1,\ldots,a_N)$. By construction this is the map into weighted projective space defined by the elements of $-K_{X_P}$; that is, the usual anti-canonical embedding. Combining this with Theorem~\ref{thm:embedding} gives the following proposition:

\begin{pro}
Given a Fano polygon $P$ isomorphic to the polyhedron of sections of a nef divisor on $\PP^2$ or $\PP^1\times \PP^1$, or isomorphic to the cone over the polyhedron of sections of a nef divisor on $\PP^1$, then $X_P$ is anti-canonically embedded as complete intersection in a weighted projective space.
\end{pro}

\begin{rem}
Note that any low codimension model obtained via the anti-canonical scaffolding of a polygon can also be obtained by studying the Hilbert series of the corresponding toric variety; by using the anti-canonical scaffolding we only obtain models already accessible by well known methods. Several examples of such models appear in~\S\ref{sec:Hilbert_Series}.
\end{rem}

%In subsequent sections we use Laurent inversion to find models for the $\QQ$-Gorenstein deformation equivalence classes of surfaces corresponding to certain Fano polygons.
%of~\cite{CK17} under `Conjecture~A'.
%Not every scaffolding gives a smooth embedding of a Fano toric variety $X_{P}$ as a complete intersection inside an ambient toric variety $Y$. 

%----------------------------------------------------------------------
\section{Low codimension constructions}
\label{sec:low_codim_models}
%----------------------------------------------------------------------

\subsection{Case $l < k+2$}
\label{sec:upto_k_plus_2}
Every surface $X^{(l)}_k$ may be exhibited as a hypersurface in a toric variety. Let $P^l_k$ denote the Fano polygon obtained as the convex hull of the points
\[
\big\{ (1,0),(0,-1),(-1,k-l),(-1,k) \big\}.
\]
Consider a scaffolding of the polygon $P_k^l$ with shape $\PP^1$ consisting of three struts:
\begin{enumerate}
\item the single point $\{(1,0)\}$;
\item the segment $[(0,-1),(0,0)]$; and
\item the segment $[(-1,k-l),(-1,k)]$.
\end{enumerate}
The polygon $P^2_4$, together with its prescribed scaffolding, is shown in Figure~\ref{fig:ex_scaff}.
%\begin{figure}
%\includegraphics{scaff_example}
%\caption{Scaffolding the polygon $P^l_k$}
%\label{fig:ex_scaff}
%\end{figure}
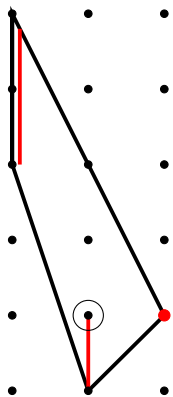
\begin{figure}[htpb]
  \centering
  \begin{tikzpicture}[transform shape]
    \begin{scope}
      \clip (-1.3,-1.3) rectangle (1.3cm,4.3cm); % Clips the picture...
      % \filldraw[fill=cyan, draw=blue] (-1,1) -- (0,1) -- (1,0) -- (1,-1) -- (0,-1) -- (-1,0) -- (-1,1); % Puts the shaded rectangle
      \draw[line width=0.5mm, black] (-1,4) -- (-1,2) -- (0,-1) -- (1,0) -- (-1,4) -- (-1,2); 
      \draw[line width=0.5mm, red] (-0.9,2) -- (-0.9,3.8); 
      % \draw[line width=0.5mm, red] (0.1,0.1) -- (0.1,0.8) -- (0.8,0.1) -- (0.1,0.1);
      \draw[line width=0.5mm, red] (0,-1) -- (0,0); 
      \foreach \x in {-7,-6,...,7}{                           % Two indices running over each
        \foreach \y in {-7,-6,...,7}{                       % node on the grid we have drawn 
          \node[draw,shape = circle,inner sep=1pt,fill] at (\x,\y) {}; % Places a dot at those points
        }
      }
      \node[draw,shape = circle,inner sep=4pt] at (0,0) {}; % Places a dot at those points
      \node[draw,shape = circle,inner sep=1.5pt,fill,red] at (1,0) {};
    \end{scope}
  \end{tikzpicture}
\caption{The scaffolding of $P^{2}_4$.}
\label{fig:ex_scaff}
\end{figure}
The weight matrix obtained via Laurent inversion from this scaffolding is:
\[
  \begin{array}{ccccc}
  y_1 & y_2 & x_1 & x_2 & x_3 \\ \midrule
    1 & 0 & 0 & 1 & 0  \\
    0 & 1 & 1 & l-k & k
  \end{array}
\]
By Theorem~\ref{thm:embedding} there is a (codimension one) embedding of the toric variety $X_{P^l_k}$ into the toric variety $Y^l_k$ defined by this matrix of weight data and the stability condition $\omega = (1,2)$.
\begin{lem}
The toric variety $Y^l_k$ is isomorphic to the rational scroll $\PP_{\PP(1,1,k)}(\cO \oplus \cO(k-l))$.
\end{lem}
The toric variety $X_{P^l_k}$ is a hypersurface given by the vanishing of $y_1y^l_2 = x_2x_3$, a section of $\cO(1,l)$ on $Y^l_k$. We now show that a general section of $\cO(1,l)$ is the blow-up of $\PP(1,1,k)$ in $l$ points.
\begin{pro}
Let $X$ be the vanishing locus of a general section of $\cO(1,l)$ on $Y^l_k$. The projection $\pi \colon Y^l_k \rightarrow \PP(1,1,k)$ maps $X$ onto $\PP(1,1,k)$ and contracts $l$ disjoint rational curves.
\end{pro}
\begin{proof}
The equation defining $X$ has the general form
\[
y_1f_{l}(y_2,x_1,x_3) + x_2g_{k}(y_2,x_1,x_3) = 0,
\]
where $f_l$,~$g_k$ are homogeneous polynomials of bi-degree $(0 , l )$ and $(0 , k)$ respectively. Therefore $X$ is a section of the projection $\pi$ except where $f_l = g_k = 0$ in $\PP_{(y_{2}:x_{1}:x_{3})}(1,1,k)$. When these two polynomials vanish the fibre of $\pi|_X$ is a $\PP^1$ contracted to a point by $\pi$. Therefore we only need to count the number of intersection points of the zero locus of $f_{l}$ and $g_{k}$.

First assume that $l < k$. Then no term of $f_l$ contains the variable $x_3$ and the vanishing locus is a collection of $l$ fibres of the projection $\PP(1,1,k) \rightarrow \PP^1$ presenting $\PP(1,1,k)$ as the cone over a rational curve of degree $k$. The vanishing locus of $g_k$ is a section of the standard projection $\PP(1,1,k) \dashrightarrow \PP^1$ and thus the two curves meet in precisely $l$ points.

Next consider the case $l=k$. The toric ambient space is $Y^l_k \cong \PP(1,1,k)\times \PP^1$. The number of points in the intersection $f_l = g_k$ is the self-intersection number of the toric divisor $x_3=0$ in $\PP(1,1,k)$, that is, $l$.

Finally consider the case $l = k+1$. As before the curve $\{g_k = 0\}$ is a section of the projection of $\PP(1,1,k)$ to $\PP^1$. The polynomial $f_{k+1} = 0$ can be written as $f_1(x_1,y_2)x_3 + h_{k+1}(x_1,y_2)$, and writing $g_k = x_3 - h_k(x_1,y_2)$, eliminate $x_3$ and solve $f_1h_k + h_{k+1} = 0$. Any solution gives a point of intersection, and thus there are $k+1=l$ such points of intersection.
\end{proof}

We also need to consider the exceptional case $B^{(k)}_k$. Consider the polygon $P_k$ defined by taking the convex hull of of the points
\[
\big\{ (1,0),(-1,-1),(-1,k) \big\}.
\]
Consider a scaffolding of the polygon $P_k$ with shape $\PP^1$ consisting of two struts:
\begin{enumerate}
\item the single point $\{(1,0)\}$; and
\item the segment $[(-1,-1),(-1,k)]$.
\end{enumerate}
Applying Laurent inversion to this scaffolding of $P_k$ obtain the toric surface $X_{P_k}$ embedded in $\PP(1,1,1,k)$ with co-ordinates $x_1,x_2,x_3,y$ via the homogeneous equation
\[
x^{k+1}_1 - x_3y = 0,
\]
that is, as a section of $\cO(k+1)$. Note that in the case $k=1$ this reproduces the Segre embedding $\PP^1\times \PP^1 \hookrightarrow \PP^3$ cut out via a section of the line bundle $\cO(2)$.

\begin{pro}
A general section of $\cO(k+1)$ on $\PP(1,1,1,k)$ is the surface $B^{(k)}_k$.
\end{pro}
\begin{proof}

The GIT presentation of $Y_{k}^{k+1}$ immediately shows that this variety is a weighted blow-up of $\PP(1,1,1,k)$ with centre $\{ y_2=x_1=x_3=0 \}$, with co-ordinates inherited from those on $Y_{k}^{(k+1)}$. Thus there are a pair of projections:
\[
\xymatrix{
Y^{(k+1)}_k \ar^{\pi_1}[d] \ar^{\pi_2}[r] & \PP(1,1,1,k) \\
\PP(1,1,k) & 
}
\]
Recall that the hypersurface $X_{k}^{(k+1)} \subset Y_{k}^{(k+1)}$ is given by the vanishing of a general section
\[
y_1f_{k+1}(y_2,x_1,x_3) - x_2g_k(y_2,x_1,x_3)= 0.
\]
This intersects the exceptional divisor $\{y_1=0\}$ in the curve $C = \{ g_k(y_2,x_1,x_3) =0 \}$ (since $x_2$ is nowhere vanishing on the exceptional divisor). The image of $X_{k}^{(k+1)}$ under $\pi_2$ is the contraction of $C$ in $X_{k}^{(k+1)}$. However the image of $C$ under $\pi_1$ is a curve in the linear system $\cO(k)$ which meets the $k+1$ points blown up by the map $\pi_1 \colon X^{(k+1)}_{k} \rightarrow \PP(1,1,k)$. Finally, observe that the push-forward of the cycle $X^{(k+1)}_k$ is a divisor in the linear system $\cO(k+1)$.
\end{proof}

Consider next those cases for which $k+2 \leq l < (k+2)^2/k$. Writing $(k+2)^2/k = k + 4 + 4/k$ there are precisely three possibilities for $l$ if $k > 3$. Consider each of these three cases in turn, noting that the behaviour of our constructions varies with the parity of $k$. Our constructions apply for all positive integers $k$, but as noted in Remark~\ref{rem:small_k}, in the cases $k=2$, and $k=4$ the general sections of the complete intersections also smooth the $\frac{1}{k}(1,1)$ singularity.

\subsection{Case $l = k+2$}
\label{sec:k_plus_2}

First consider the case $k = 2m$ for some $m \in \ZZ_{> 2}$. Consider the polygon $P^{k+2}_k$ given by the convex hull of the points
\[
\big\{ (-1,-1) , (1,-1) , (-1,m) , (1,m) \big\}.
\]
The case $m=3$ is shown in Figure~\ref{fig:Pkplus2even} equip with its anti-canonical scaffolding.
\begin{figure}[htbp]
  \centering
  \begin{tikzpicture}[transform shape]
    \begin{scope}
      \clip (-1.3,-1.3) rectangle (1.3cm,3.3cm); % Clips the picture...
      % \filldraw[fill=cyan, draw=blue] (-1,1) -- (0,1) -- (1,0) -- (1,-1) -- (0,-1) -- (-1,0) -- (-1,1); % Puts the shaded rectangle
      \draw[line width=0.5mm, black] (-1,-1) -- (1,-1) -- (1,3) -- (-1,3) -- (-1,-1);
      \draw[line width=0.5mm, red] (-0.9,-0.9) -- (0.9,-0.9) -- (0.9,2.9) -- (-0.9,2.9) -- (-0.9,-0.9);
      %\draw[line width=0.5mm, red] (-0.9,1.9) -- (0.9,1.9) -- (0.9,0.1) -- (-0.9,0.1) -- (-0.9,1.9) -- (0.9,1.9); 
      % \draw[line width=0.5mm, red] (0.1,0.1) -- (0.1,0.8) -- (0.8,0.1) -- (0.1,0.1);
  %    \draw[line width=0.5mm, red] (-0.9,0.9) -- (1.9,0.9) -- (1.9,-0.9) -- (-0.9,-0.9) -- (-0.9,0.9) -- (1.9,0.9) ; 
      \foreach \x in {-7,-6,...,7}{                           % Two indices running over each
        \foreach \y in {-7,-6,...,7}{                       % node on the grid we have drawn 
          \node[draw,shape = circle,inner sep=1pt,fill] at (\x,\y) {}; % Places a dot at those points
        }
      }
      \node[draw,shape = circle,inner sep=4pt] at (0,0) {}; % Places a dot at those points
    \end{scope}
  \end{tikzpicture}
\caption{The scaffolding of $P^{(8)}_6$.}
\label{fig:Pkplus2even}
\end{figure}
Following the Laurent inversion construction (or otherwise) the anti-canonical embedding maps
\[
X_{P^{k+2}_2} \hookrightarrow \PP(1,1,1,1,m).
\]
This coincides with the model suggested in~\S\ref{sec:Hilbert_Series}. In particular the image of this embedding is a codimension two complete intersection given by the vanishing of a section of the split bundle $E := \cO(2)\oplus\cO(m+1)$. In fact, one can show explicitly that the vanishing of a section of $E$ is precisely a surface $X^{(k+2)}_k$.
\begin{pro}
The minimal resolution of the vanishing of any section of $E$ on $Y^{(k+2)}_k := \PP(1,1,1,1,m)$ is the blow-up of $\PP^1\times \PP^1$ in $k+2$ points.
\end{pro}
\begin{proof}
Let $x_i$,~$1 \leq i \leq 4$ and $y$ by the co-ordinates on $Y^{(k+2)}_k$ and consider the vanishing locus $V := \{s_2 = 0\}$ of a section of $\cO(2)$ on $Y^{(k+2)}_k$. The section $s_2$ is represented by a homogeneous polynomial with no term containing the variable $y$. Therefore $V$ is isomorphic to a cone over the Segre embedding of $\PP^1\times \PP^1$. The complement of the point $\{x_1=x_2=x_3=x_4=0\}$ in $V$ is the total space of $\cO(m,m)$ on $\PP^1\times\PP^1$.

Let $W$ be the vanishing locus of $\{s_{m+1} = 0\}$, a homogeneous polynomial of degree $m+1$. This has the general form
\[
s_{m+1} = yf_1(x_1,\ldots x_4) +  f_{m+1}(x_1,\ldots x_4).
\]
Consider the projection $X := V \cap W \dashrightarrow \PP^1\times \PP^1$ which contracts precisely those curves fibering over the points $f_1 = f_{m+1} = 0$. Sections of $\cO(a)$ on $\PP^3$, for any $a \in \mathbb{N}$ pull back to sections of $\cO(a,a)$ on $\PP^1\times\PP^1$ under the Segre embedding and thus the locus $f_1=f_{m+1}=0$ consists of precisely $2(m+1) = k+2$ points on a curve in the linear system of $\cO(1,1)$, and so up to a linear co-ordinate change, consists of $k+2$ points on the diagonal $\Delta$ of $\PP^1\times\PP^1$.

In fact this projection factors through the blow-up of $Y^{(k+2)}_k$ at the point $\{x_1 = \ldots = x_4 = 0\}$, resolving the indeterminacy of the projection and resolving the $1/k(1,1)$ singularity of the surface $X$. This therefore exhibits $k+2$ disjoint lines on the minimal resolution of $X$ and contracting these yields the surface $\PP^1\times \PP^1$. By Lemma~\ref{lem:alternate_blowup}, $X$ is the blow-up of $\PP(1,1,k)$ in $k+2$ points.
\end{proof}

Assume instead that $k = 2m-1$ for some $m \in \ZZ_{\geq 1}$. This case closely generalises the surface $dP_6$ in the case $k=1$. The case $k=3$ appears in~\cite{RS03} and has degree $10/3$. There Reid--Suzuki observe that the surface $X^{(5)}_3$ naturally embeds in codimension four. However we construct a codimension two embedding into a toric variety via Laurent inversion analogous to the embedding of $dP_6$ into the fourfold $\PP^2 \times \PP^2$. 
\medskip

The case $k=1$ is nothing other than the usual construction of $dP_6$ as a codimension two complete intersection in $\PP^2\times \PP^2$, the \emph{ancestral Tom} of Brown--Reid--Stevens~\cite{browntutorial}. Similarly there is a codimension four Segre type embedding of $Y^{(k+2)}_k$ into $\PP(1^4,m^4,k)$ (where superscripts indicate repeated weights). In the case $k=1$ there is also an embedding into the \emph{ancestral Jerry} ($\PP^1\times \PP^1 \times \PP^1$). This construction does not appear to generalise to other values of $k$.
\medskip

Consider the polygon $P^{(k+2)}_k$ given as the convex hull of the points
\[ \big\{ (0,-1), (m,-1), (m,m-1), (m-1,m), (-1,m), (-1,0) \big\}, \] 
together with the scaffolding shown in Figure~\ref{fig:Pkplus2k} with shape $\mathbb{P}^{1} \times \mathbb{P}^{1}$.
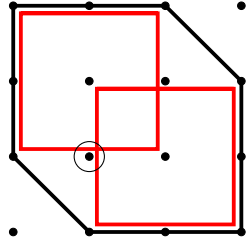
\begin{figure}[htpb]
  \centering
  \begin{tikzpicture}[transform shape]
    \begin{scope}
      \clip (-1.3,-1.3) rectangle (2.3cm,2.3cm); % Clips the picture...
      % \filldraw[fill=cyan, draw=blue] (-1,1) -- (0,1) -- (1,0) -- (1,-1) -- (0,-1) -- (-1,0) -- (-1,1); % Puts the shaded rectangle
      \draw[line width=0.5mm, black] (-1,0) -- (0,-1) -- (2,-1) -- (2,1) -- (1,2) -- (-1,2) -- (-1,0); 
      \draw[line width=0.5mm, red] (-0.9,1.9) -- (0.9,1.9) -- (0.9,0.1) -- (-0.9,0.1) -- (-0.9,1.9) -- (0.9,1.9); 
      % \draw[line width=0.5mm, red] (0.1,0.1) -- (0.1,0.8) -- (0.8,0.1) -- (0.1,0.1);
      \draw[line width=0.5mm, red] (0.1,0.9) -- (1.9,0.9) -- (1.9,-0.9) -- (0.1,-0.9) -- (0.1,0.9) -- (1.9,0.9) ; 
      \foreach \x in {-7,-6,...,7}{                           % Two indices running over each
        \foreach \y in {-7,-6,...,7}{                       % node on the grid we have drawn 
          \node[draw,shape = circle,inner sep=1pt,fill] at (\x,\y) {}; % Places a dot at those points
        }
      }
      \node[draw,shape = circle,inner sep=4pt] at (0,0) {}; % Places a dot at those points
    \end{scope}
  \end{tikzpicture}
\caption{The scaffolding used to construct $X^{(k+2)}_{k}$ in the case $k=3$.}
\label{fig:Pkplus2k}
\end{figure}

This scaffolding induces a toric embedding of $X_{P^{k+2}_k}$ into a toric variety $Y^{(k+2)}_k$ defined by the weight matrix

\[
  \begin{array}{cccccc}
  x_1 & x_2 & y_1 & y_2 & z_1 & z_2 \\ \midrule
    1 & 1 & 0 & 0 & m-1 & m \\
    0 & 0 & 1 & 1 & m & m-1
  \end{array}
\]
together with stability condition $\omega = (1,1)$. The fourfold $Y^{(k+2)}_k$ determined by this data is a $\QQ$-factorial Fano variety. The surface $X_{P^{k+2}_k}$ is a codimension two complete intersection defined by the vanishing of the polynomials
\begin{align*}
x^m_1y^m_1 - x_2z_1, \quad \text{and} \quad x^m_1y^m_1 - y_2z_2.
\end{align*}
In particular $X_{P^{k+2}_k}$ admits a flat deformation to the vanishing locus $X$ of a general section of the split bundle $E := \cO(m,m)^{\oplus 2}$.

\begin{pro}
\label{pro:k_plus_2}
The minimal resolution of the vanishing of any section of $E$ on $Y^{(k+2)}_k$ is the blow-up of $\PP^1\times \PP^1$ in $k+2$ points on the diagonal $\Delta$ (the surface $S_k$ of Lemma~\ref{lem:alternate_blowup}). Moreover this resolution contracts the strict transform of the diagonal of $\PP^1\times \PP^1$.
\end{pro}
\begin{proof}
Any section of the split bundle $E$ is defined by the pair of equations,
\begin{align*}
z_1f_{1,0}(x_1,x_2) + z_2g_{1,0}(y_1,y_2) + f_{m,m}(x_1,x_2,y_1,y_2) = 0, \\
z_1h_{1,0}(x_1,x_2) + z_2k_{1,0}(y_1,y_2) + g_{m,m}(x_1,x_2,y_1,y_2) = 0,
\end{align*}
where subscripts of polynomials indicate degree in the homogeneous co-ordinate ring of $Y^{(k+2)}_k$. There is an obvious projection
\[
\pi_k \colon Y^{(k+2)}_k \dashrightarrow \PP^1\times\PP^1
\]
obtained by projecting out $z_1$ and $z_2$. This projection is defined away from the loci $\{x_1=x_2=0\}$ and $\{y_1=y_2 = 0\}$. These loci meet the vanishing locus of every section of $E$ at the point $x_1=x_2=y_1=y_2=0$ (since the loci $\{x_1=x_2=z_2=0\}$ and $\{y_1=y_2=z_1=0\}$ are unstable). As in the case of $k \in 2\ZZ$ the projection $\pi_k$ contracts a number of curves. These curves are defined by two conditions; first we need the matrix 
\[
\begin{pmatrix}
f_1 & g_1 \\
h_1 & k_1
\end{pmatrix}
\]
to drop rank. This condition defines an equation in $\cO(1,1)$ on $\PP^1\times\PP^1$. Second we need this locus to intersect the surface $X$. This occurs when the following matrix also drops rank
\[
\begin{pmatrix}
f_{m,m} & f_1 \\
g_{m,m} & h_1
\end{pmatrix}.
\]
The first equation determines a section of $\cO(1,1)$ which is assumed to be the diagonal $\Delta$ in $\PP^1\times\PP^1$. The second equation defines an equation in $\cO(m+1,m)$ on $\PP^1\times\PP^1$. Taking the intersection note that the fibre of $\pi_k$ over $2m+1 = k+2$ points of $\Delta$ contains an exceptional curve. Over every point away from $\Delta$, the fibre of $\pi_k$ consists of a single point.
\end{proof}

\begin{cor}
	General sections of $E$ are surfaces in the family $X^{(k+2)}_k$.
\end{cor}
\begin{proof}
	Contracting the strict transform of the diagonal in $S_k$ we obtain a surface in the family $X^{(k+2)}_k$ via Lemma~\ref{lem:alternate_blowup}.
\end{proof}

\subsection{Case $l = k+3$}
\label{sec:k_plus_3}

Again consider the (easier) case of $k = 2m$ for some $m \in \ZZ_{\geq 1}$. In the case $l=k+2$ and $k \in 2\ZZ_{\geq 1}$ the anti-canonical embedding of $X^{(k+2)}_k$ is codimension two and there are explicit lines making divisorial contractions to $\PP^1\times \PP^1$. It is therefore expected that the $l=k+3$ case will be anti-canonically embedded as a hypersurface in a weighted projective space obtained by a linear projection from $X^{(k+2)}_k \subset \PP(1,1,1,1,m)$. We demonstrate this using Laurent inversion.

Consider the polygon $P^{(k+3)}_k$ with vertices
\[
\big\{ (-1,-1),(-1,m+1),(m+1,-1) \big\}.
\]
Applying Laurent inversion to $P^{(k+3)}_k$ with the anti-canonical scaffolding with shape $\PP^2$ obtain the variety $Y^{(k+3)}_k := \PP(1,1,1,m)$ with homogeneous co-ordinates $x_i$,~$1 \leq i \leq 3$ and $y$. The toric surface $X_{P^{(k+3)}_k}$ is given by the vanishing of the section $x_1^{m+2} - x_2x_3y$ of $\cO(m+2)$. The surfaces $X^{(k+2)}_k$ are obtained from these hypersurfaces by the simplest kind of unprojection, from codimension one to codimension two. Explicitly assume that the equation defining a general section $X$ of $\cO(m+2)$ in $\PP(1,1,1,m)$ has the form
\[
Ay - Bx_3 = 0,
\]
where $A$ has degree $2$ and $B$ has degree $m+1$. Introducing the unprojection variable $s$ obtain the equations
\begin{align*}
sx_3 = A \quad \text{and} \quad sy = B
\end{align*}
in $\PP(1,1,1,1,m)$ of degrees $2$ and $m+1$ respectively. In particular note that the projection from $X^{(k+2)}_k$ to $X^{(k+3)}_k$ is a blow-up of a single smooth point.

Now suppose $k=2m-1$ for an integer $m \in \ZZ_{\geq 1}$. Here our surfaces come anti-canonically embedded in codimension three, as the cases $k=1$ ($dP_5$), $k=3$ (see~\cite{RS03}) and the Hilbert series calculations in~\S\ref{sec:Hilbert_Series} suggest. It is therefore reasonable to consider the Pfaffians of a $5\times 5$ matrix. However, again following the path suggested by Laurent inversion, obtain a hypersurface embedding of $X^{(k+3)}_k$ into a toric variety. %This embedding generalises the well known construction of $dP_5$ as a hypersurface in $\PP^2\times \PP^1$ by a section of $\cO(2,1)$, where the first projection (to $\PP^2$ contracts curves whose image lies in the intersection of two conics, that is, over four points). 

The embedding $X^{k+3}_k \hookrightarrow Y^{(k+3)}_k$ is the most interesting application of Laurent inversion in this paper. Let $P^{(k+3)}_k$ be the convex hull of vertices
\[
\big\{ (-1,-1),(-1,m),(m-1,m),(m,m-1),(m,-1) \big\},
\]
and cover $P^{(k+3)}_k$ by a pair of struts with shape $\PP^1\times \PP^1$ as shown in Figure~\ref{fig:Pkplus3k}.

\begin{figure}[htpb]
  \centering
  \begin{tikzpicture}[transform shape]
    \begin{scope}
      \clip (-1.3,-1.3) rectangle (2.3cm,2.3cm); % Clips the picture...
      % \filldraw[fill=cyan, draw=blue] (-1,1) -- (0,1) -- (1,0) -- (1,-1) -- (0,-1) -- (-1,0) -- (-1,1); % Puts the shaded rectangle
      \draw[line width=0.5mm, black] (-1,-1) -- (2,-1) -- (2,1) -- (1,2) -- (-1,2) -- (-1,-1); 
      \draw[line width=0.5mm, red] (-0.9,1.9) -- (0.9,1.9) -- (0.9,0.1) -- (-0.9,0.1) -- (-0.9,1.9) -- (0.9,1.9); 
      % \draw[line width=0.5mm, red] (0.1,0.1) -- (0.1,0.8) -- (0.8,0.1) -- (0.1,0.1);
      \draw[line width=0.5mm, red] (-0.9,0.9) -- (1.9,0.9) -- (1.9,-0.9) -- (-0.9,-0.9) -- (-0.9,0.9) -- (1.9,0.9) ; 
      \foreach \x in {-7,-6,...,7}{                           % Two indices running over each
        \foreach \y in {-7,-6,...,7}{                       % node on the grid we have drawn 
          \node[draw,shape = circle,inner sep=1pt,fill] at (\x,\y) {}; % Places a dot at those points
        }
      }
      \node[draw,shape = circle,inner sep=4pt] at (0,0) {}; % Places a dot at those points
    \end{scope}
  \end{tikzpicture}
\caption{The scaffolding used to construct $X^{(k+3)}_{k}$ in the case $m=2$.}
\label{fig:Pkplus3k}
\end{figure}
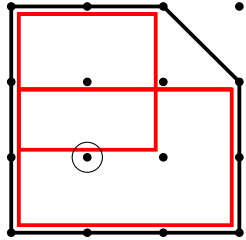

\noindent This scaffolding determines a toric variety $Y^{(k+3)}_k$ with matrix of weight data
\[
  \begin{array}{cccccc}
  x_1 & x_2 & y_1 & y_2 & z_1 & z_2 \\ \midrule
    1 & 1 & 0 & 1 & m-1 & m \\
    0 & 0 & 1 & 1 & m & m-1
  \end{array}
\]
and stability condition $\omega = (1,1)$. The surface $X_{P^{k+3}_k}$ is a codimension two complete intersection defined by the vanishing of the polynomials
\begin{align*}
x^m_1y^m_1 - x_2z_1 \quad \text{and} \quad x^{m+1}_1y^m_1 - y_2z_2.
\end{align*}

\noindent Thus $X_{P^{k+3}_k}$ admits a flat deformation to a general section of the vector bundle $E := \cO(m,m)\oplus\cO(m+1,m)$. Note that the fourfold $Y^{(k+3)}_k$ is \emph{not} $\QQ$-factorial, since $Y^{(k+3)}_k$ contains the point $\{x_1=x_2=y_1=z_1=z_2=0\}$. Also note that the toric subvariety $X_{P^{(k+3)}_k}$ meets this point, although the general section of the split bundle $E$ does not.

\begin{pro}
\label{pro:k_plus_3}
The minimal resolution of the vanishing of any section of $E$ on $Y^{(k+3)}_k$ is the blow-up of $\PP^2$ in $k+4$ points lying on a conic. Moreover the resolution contracts the strict transform of the conic.
\end{pro}
\begin{proof}

Similarly to the case $l=k+2$ there is an obvious projection
\[
\pi_k \colon Y^{(k+3)}_k \dashrightarrow \FF_1
\]
onto the Hirzebruch surface $\FF_1$ with homogeneous co-ordinates $x_1,x_2,y_1$, and $y_2$. Following the method used in the proof of Proposition~\ref{pro:k_plus_2} form an expression for a general section of $E$,

\begin{align*}
z_1f_{1,0} + z_2f_{0,1} + f_{m,m} = 0 \\
z_1f_{2,0} + z_2f_{1,1} + f_{m+1,m} = 0
\end{align*}
where $f_{i,j}$ denotes a polynomial of bidegree $(i,j)$ in the homogeneous co-ordinate ring of $\FF_1$. The rational map $\pi_k$ is undefined along $\{x_1=x_2=0\}$ and along $\{y_1=y_2=0\}$. These loci meet $Y^{(k+3)}_k$ at the point $\{x_1=x_2=y_1=y_2=0\}$. Restricting the defining equations of $Y^{(k+3)}_k$ to $\{x_1=x_2=0\}$ obtain the equations
\begin{align*}
z_2y_1 + y_2^m=0 \quad \text{and} \quad y_2z_2=0.
\end{align*}
Noting that the locus $\{x_1=x_2=z_2=y_2=0\}$ is empty in $Y^{(k+3)}_k$, the equations are only satisfied when $y_1=y_2=0$. A similar calculation shows $Y^{(k+3)}_k$ meets the locus $\{y_1=y_2=0\}$ at this point. Next consider the conditions required for a given fibre of $\pi_k$ to contain a line. There is an equation with bidegree $\cO(2,1)$ on $\FF_1$ given by the vanishing of the determinant of the matrix
\[
\begin{pmatrix}
f_{1,0} & f_{0,1} \\
f_{2,0} & f_{1,1}
\end{pmatrix}.
\]
There is also an equation of bidegree $\cO(m+1,m+1)$ given by the vanishing of the determinant of the matrix
\[
\begin{pmatrix}
f_{m,m} & f_{0,1} \\
f_{m+1,m} & f_{1,1}
\end{pmatrix}.
\]
The intersection form on $\FF_1$ in the basis of $\Pic(\FF_1)$ determined by the weight matrix defining $Y^{(k+3)}_k$ has matrix 
\[
\begin{pmatrix}
0 & 1\\
1 & -1
\end{pmatrix}.
\]
Thus the intersection product $\langle (m+1,m+1), (2,1) \rangle$ is equal to $2m+2 = k+3$ and the projection $\pi_k$ contracts precisely $k+3$ curves on fibering over a section of $\cO(2,1)$.
\end{proof}

\begin{cor}
	General sections of $E$ are surfaces in the family $X^{(k+3)}_k$.
\end{cor}

\begin{proof}
	By Lemma~\ref{lem:alternate_blowup_2}, by contracting the strict transform of the conic obtain a surface in the family $X^{(k+3)}_k$.
\end{proof}

In the case $m=1$, this reduces to the case of $dP_5 \subset \PP^2\times \PP^1$  cut out by a section of $\cO(2,1)$. Note however that we had to add an additional column $(1,1)$ to the weight matrix, and an line bundle $\cO(1,1)$ before this construction generalises to arbitrary values of $m$.
\medskip

In~\cite{RS03} Reid--Suzuki observe that (similarly to $dP_5$) the surface $X^{(6)}_3$ embeds in codimension three via a system of Pfaffians of a $5 \times 5$ matrix. In fact such a construction works in general, and corresponds to the anti-canonical scaffolding of $P^l_k$ shown in Figure~\ref{fig:anti_canonical_cod3}. Indeed, in~\S\ref{sec:kplus4} there is a codimension two model of the surface $X^{(k+4)}_k$ and, making a suitable unprojection from this surface, it is possible to recover the surface $X^{(k+3)}_k \subset \PP(1,1,1,m,m,k)$.

\begin{figure}[htpb]
  \centering
  \begin{tikzpicture}[transform shape]
    \begin{scope}
      \clip (-1.3,-1.3) rectangle (2.3cm,2.3cm); % Clips the picture...
      % \filldraw[fill=cyan, draw=blue] (-1,1) -- (0,1) -- (1,0) -- (1,-1) -- (0,-1) -- (-1,0) -- (-1,1); % Puts the shaded rectangle
      \draw[line width=0.5mm, black] (-1,-1) -- (2,-1) -- (2,1) -- (1,2) -- (-1,2) -- (-1,-1);
      \draw[line width=0.5mm, red] (-0.9,-0.9) -- (1.9,-0.9) -- (1.9,0.9) -- (0.9,1.9) -- (-0.9,1.9) -- (-0.9,-0.9);  
      %\draw[line width=0.5mm, red] (-0.9,1.9) -- (0.9,1.9) -- (0.9,0.1) -- (-0.9,0.1) -- (-0.9,1.9) -- (0.9,1.9); 
      % \draw[line width=0.5mm, red] (0.1,0.1) -- (0.1,0.8) -- (0.8,0.1) -- (0.1,0.1);
  %    \draw[line width=0.5mm, red] (-0.9,0.9) -- (1.9,0.9) -- (1.9,-0.9) -- (-0.9,-0.9) -- (-0.9,0.9) -- (1.9,0.9) ; 
      \foreach \x in {-7,-6,...,7}{                           % Two indices running over each
        \foreach \y in {-7,-6,...,7}{                       % node on the grid we have drawn 
          \node[draw,shape = circle,inner sep=1pt,fill] at (\x,\y) {}; % Places a dot at those points
        }
      }
      \node[draw,shape = circle,inner sep=4pt] at (0,0) {}; % Places a dot at those points
    \end{scope}
  \end{tikzpicture}
\caption{The anti-canonical scaffolding of $P^{k+3}_{k}$ in the case $k=3$.}
\label{fig:anti_canonical_cod3}
\end{figure}
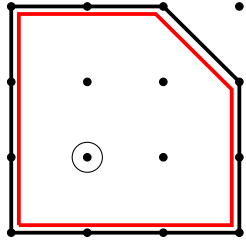

\noindent Following the argument used in~\cite{RS03} this model works, taking a matrix
\[
\begin{pmatrix}
x_1 & x_2 & b_{14} & b_{15} \\
 & x_3 & b_{24} & b_{25}  \\
 & & b_{34} & b_{35} \\
 & & & z
\end{pmatrix}
\text{ of degrees }
\begin{pmatrix}
1 & 1 & m & m \\
 & 1 & m & m  \\
 & & m & m \\
 & & & k
\end{pmatrix}
\]
where $x_i$,~$1 \leq i \leq 3$ and $z$ are the co-ordinates on $\PP(1,1,1,m,m,k)$ of degrees $1$ and $k$ respectively.

\subsection{\protect{Case $l = k+4$}}
\label{sec:kplus4}

The Hilbert series calculations in~\S\ref{sec:Hilbert_Series} suggest a model for $X^{(k+4)}_k$ in weighted projective space of codimension $\leq 2$ for all $k \in \ZZ_{\geq 1}$. These models should coincide with the model suggested by Laurent inversion applied to the anti-canonical scaffolding of a polygon associated to a toric degeneration of $X^{(k+4)}_k$. Figure~\ref{fig:Pkplus4} gives an example of polygons $P^{k+4}_k$ for each parity of $k$.

\begin{figure}[htpb]
  \centering
  \begin{tikzpicture}[scale = 0.5, transform shape]
    \begin{scope}
      \clip (-1.3,-2.3) rectangle (5.3cm,4.3cm); % Clips the picture...
      \draw[line width=0.5mm, black] (-1,-2) -- (5,-2) -- (-1,4) -- (-1,-2);
      \draw[line width=0.5mm, red] (-0.8,-1.8) -- (4.5,-1.8) -- (-0.8,3.5) -- (-0.8,-1.8);
      \foreach \x in {-7,-6,...,7}{                           % Two indices running over each
        \foreach \y in {-7,-6,...,7}{                       % node on the grid we have drawn 
          \node[draw,shape = circle,inner sep=1pt,fill] at (\x,\y) {}; % Places a dot at those points
        }
      }
      \node[draw,shape = circle,inner sep=4pt] at (0,0) {}; % Places a dot at those points
    \end{scope}
  \end{tikzpicture}
  \qquad
  \qquad
  \begin{tikzpicture}[scale = 0.5, transform shape]
  % Square side length k+1 = 4, m = 2 start at (-1,-2), 
    \begin{scope}
      \clip (-1.3,-2.3) rectangle (5.3cm,4.3cm); % Clips the picture...
      \draw[line width=0.5mm, black] (-1,-2) -- (5,-2) -- (5,4) -- (-1,4) -- (-1,-2);
      \draw[line width=0.5mm, red] (-0.8,-1.8) -- (4.8,-1.8) -- (4.8,3.8) -- (-0.8,3.8) -- (-0.8,-1.8);
      \foreach \x in {-7,-6,...,7}{                           % Two indices running over each
        \foreach \y in {-7,-6,...,7}{                       % node on the grid we have drawn 
          \node[draw,shape = circle,inner sep=1pt,fill] at (\x,\y) {}; % Places a dot at those points
        }
      }
      \node[draw,shape = circle,inner sep=4pt] at (0,1) {}; % Places a dot at those points
    \end{scope}
  \end{tikzpicture}
\caption{The anti-canonical scaffolding for $X^{(k+4)}_{k}$ in the case $k=4$ and $k=5$.}
\label{fig:Pkplus4}
\end{figure}
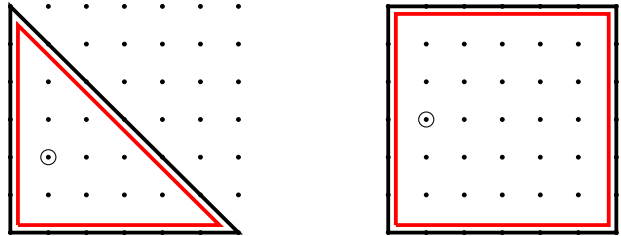

It is routine to verify that the singularities of a general section of each of these complete intersections is as expected. For $k = 2m-1$ where $m \in \ZZ_{\geq 1}$, obtain the model
\[
X_{k+1,k+1} \subset \PP(1,1,m,m,k),
\]
which, applying Theorem~\ref{thm:quasismooth_codim_2}, is a quasismooth codimension two complete intersection. From this it is easy to verify that it has the correct singularities.

Contrary to previous subsections, the case $k=2m$ for some $m \in \ZZ_{\geq 1}$ is more complicated. The model
\[
X_{k+2} \subset \PP(1,1,m,m+1),
\]
with co-ordinates~$x_1$,~$x_2$,~$y$ and $z$, is not quasismooth. Indeed, choosing a general $f \in \Gamma(\cO(k+2))$ the affine variety $\{f=0\} \subset \AA^4$ is singular along the line $L = \{x_1=x_2=z=0\}$. Setting $y=y_0$ the lowest order terms of $f$ have degree two and the singularity in the affine slice $y=y_0$ is an ordinary double point. Taking the quotient by $\GG_m$ maps $L \subset \AA^4$ to a $\frac{1}{m}(1,1,1)$ singularity. Considering how this group action acts on $\{f=0\}$, note the hypersurface in $\PP(1,1,m,m+1)$ defined by $f$ has a single singular point of type $\frac{1}{2m}(1,1)$, as expected.

%----------------------------------------------------------------------
\section{Classifying Root Systems}
\label{sec:root_systems}
%----------------------------------------------------------------------

This section is devoted to the proof of Theorem~\ref{thm:root_systems}. In particular we identify each root system of $(-2)$-classes in $\omega^\bot \subset \Pic(X^{(l)}_k)$ where $\omega$ is the canonical class of $X^{(l)}_k$. This section is a direct generalisation of~\cite[\S $25$]{Manin86}. Recall that Theorem~\ref{thm:root_systems} associates each surface $X^{(l)}_k$ to a root system as follows:

\[
\begin{array}{c||c|c|c|c|c|c}
l = (k+1)^2/k - d & 2 & \ldots & k+1 & k+2 & k+3 & k+4  \\
\hline
R & A_1 & \ldots & A_k & A_{k+1} \times A_1 & A_{k+3} & D_{k+4}
\end{array}
\]
\medskip
\begin{dfn}
\label{dfn:lattice}
Given $k \in \ZZ_{> 0}$, and $2 \leq l \leq k+4$, let $N^l_k$ be the lattice $\ZZ^{l+1}$ with standard basis $\{\ell_0 \ldots, \ell_l\}$. Fix a scalar product $(-,-)$ on $N^l_k$ by setting
\begin{enumerate}
\item $(\ell_0,\ell_0) = k$;
\item $(\ell_i,\ell_i) = -1$ for $1 \leq i \leq l$;
\item $(\ell_i,\ell_j) = 0$ for $i \neq j$.
\end{enumerate}
Fix the class 
\[
\omega = -\frac{(k+2)}{k} \ell_0 + \ell_1 + \ldots + \ell_l
\]
in $N^l_k \otimes \QQ$.
\end{dfn}

\begin{lem}
The lattice $\Pic(X^{(l)}_k)$, together with basis 
\[
\{\pi^\star\cO_{\PP(1,1,k)}(k),\cO(E_1),\ldots, \cO(E_l)\},
\]
where $\pi$ is the contraction of disjoint $(-1)$-curves $X^{(l)}_k \rightarrow \PP(1,1,k)$, and the usual intersection product, is isomorphic to $N^l_k$ as a based lattice with scalar product.
\end{lem}
\begin{proof}
This has an identical proof to~\cite[Proposition~$25.1$]{Manin86}. Recall that the Picard group $\Pic(\PP(1,1,k))$ is generated by $\cO(k)$ and $\pi^\star\cO(k)$ has self-intersection $k$.
\end{proof}

%Having defined the lattice $N^l_r$ together with its scalar product there is no further geometric input into the root system.
\begin{dfn}
\label{dfn:root_system}
Let $R^l_k$ denote the set of vectors $\ell \in N^l_k$ such that
\begin{align*}
(\ell,\ell) = -2 \quad \text{and} \quad (\ell,\omega_k) = 0.
\end{align*}
\end{dfn}

\begin{pro}
\label{pro:root_systems}
The set $R^l_k \subset \omega^\bot$ is a root system. In the case that $l \geq k+2$ this is a root system in the vector space $\omega^\bot\otimes_\ZZ \RR \subset {N^l_k}\otimes_\ZZ \RR$. In the case that $2 \leq l < k+2$, $R^l_k$ spans a hyperplane in $\omega^\bot\otimes_\ZZ \RR$.
\end{pro}
\begin{proof}
The proof follows ~\cite{Manin86}. First compute the length of a vector orthogonal to $\omega$ in $N^l_k$, noting that
\[
\left(\omega,a\omega + \sum\limits_{i=1}^{l}{b_i\ell_i}\right) = \left(\frac{(k+2)^2}{k}-l\right)a - \sum\limits_{i=1}^{l}{b_i}.
\]
Thus a vector lies in $\omega^\bot$ if and only if
\[
 \left(\frac{(k+2)^2}{k}-l\right)a = \sum\limits_{i=1}^{l}{b_i}.
\]
The length of such a vector is then equal to
\[
\left(\frac{(k+2)^2}{k}\right)a^2 - 2a\sum\limits_{i=1}^{l}{b_i} - \sum\limits_{i=1}^{l}{b^2_i} = \frac{-k}{(k+2)^2-lk}\left(\sum\limits_{i=1}^{l}{b_i}\right)^2 - \sum\limits_{i=1}^{l}{b^2_i}.
\]
Recalling that $l \leq k+4$ for any $k > 3$, and that in the exceptional case $k=3$ and $l=k+5$, the intersection form is negative-definite on $\omega^\bot$ for all possible pairs $(l,k)$. Let $V$ be a finite-dimensional vector space and let $R \subset V$ be a finite set. $R$ is root system $R \subset V$ if it satisfies the following properties:
\begin{enumerate}
	\item $R$ is a spanning set of $V$;
	\item the only scalar multiples of a root $x \in R$ are $\pm x$;
	\item the set $R$ is closed under reflection;
	\item for any $x$ and $m$ in $R$, $2(x,m)/(x,x)$ is an integer.
\end{enumerate}
The vectors $\ell_i - \ell_j$, $i \neq j$, span a hyperplane in $\omega^\bot$ and all lie in $R^l_k$. In the case $l \geq k+2$ the vector $\ell_0 - \ell_1 - \ldots - \ell_l$ is also a root and jointly these vectors span $\omega^\bot$. Consequently setting $V$ to be the hyperplane spanned by the $\ell_i - \ell_j$ if $l < k+2$ and $\omega^\bot$ otherwise, it follows that $R^l_k$ spans $V$.

All elements in $R^l_k$ have length~$2$ by definition and so property (ii) is automatic. Similarly $R^l_k$ is finite since it is comprised of lattice vectors of fixed length. To verify property (iii) it is required to check that
\[
x + (x,m)m
\]
is in $R^l_k$ for any $x$ and $m$ in $R^l_k$. This is obvious since length and orthogonality to $\omega$ are preserved by this reflection. Property~(iv) is also clear as all the roots have length $2$.
\end{proof}

In the cases for which $2 \leq l < k+2$ the root system is easy to identify, since the only possible roots have the form $\ell_i - \ell_j$, where $i \in \ZZ_{> 0}$,~$j \in \ZZ_{> 0}$, $i \neq j$. These vectors give the standard presentation of the root system $A_{l-1}$. In these cases the only $(-1)$-curves disjoint from the singular locus are the exceptional curves of the $l$ blow-ups of $\PP(1,1,k)$, and the Weyl group associated to this root system is the symmetric group of this set of exceptional curves.

Consider the case $k+2 \leq l \leq k+4$. To classify the root systems $R^l_k$ first identify a (large) subsystem. 
\begin{pro}
\label{pro:big_subsystem}
In the case $l=k+4$ a collection of roots is obtained from Table~\ref{fig:roots} by reversing signs and permuting the $b_i$ in all possible ways. The Cartan matrices of these roots are as tabulated in Theorem~\ref{thm:root_systems}. There are analogous collections roots in the cases $l=k+2$ and $l=k+3$ obtained by shortening Table~\ref{fig:roots}.
\end{pro}

\begin{figure}[htbp]
\[
\begin{array}{ccccccc}
a & b_1 & b_2 & \cdots & b_{k+2} & b_{k+3} & b_{k+4} \\ \hline
0 & 1 & -1 & \cdots & 0 & 0 & 0 \\
1 & 1 & 1 & \cdots & 1 & 0 & 0
\end{array}
\]
\caption{Table of the roots of $R^l_k$.}
\label{fig:roots}
\end{figure}
\begin{proof}
Compute the number of roots obtained from Table~\ref{fig:roots} (and its analogues). In each case
\[
|R^l_k| =
\begin{cases}
 (k+2)(k+1)+2, & \text{if } l=k+2; \\
 (k+4)(k+3), & \text{if } l=k+3; \\
 2(k+4)(k+3), & \text{if } l=k+4.
\end{cases}
\]
It is also easy to verify that these collections form a root system, and that a basis is given by the collection
\[
\Delta := \{\ell_{i+1}-\ell_i : 1 \leq i < l\} \cup \{\ell_0 + \cdots + \ell_{k+2}. \}
\]
Note this system has an obvious $A_{l-1}$ subsystem consisting of roots $\ell_i-\ell_j$ for $i \neq j$. In the case $l=k+2$ there are only two additional roots and we obtain the system $A_{k+1} \times A_1$. In the cases $l=k+3$ and $l=k+4$ nte that
\[
(\ell_0 + \cdots + \ell_{k+2},\ell_{i+1}-\ell_i) = 0
\]
unless $i = k+2$ or $i=k+3$. By computing the Cartan matrix of these roots identify these root systems with those enumerated in Theorem~\ref{thm:root_systems}.
\end{proof}

It still remains to verify that the roots obtained in Proposition~\ref{pro:big_subsystem} are all the roots of $R^l_k$. To do this compute the \emph{index of connectedness} of each $R^l_k$, see~\cite{Manin86}. The index of connectedness of a root system $R$ in a Euclidean vector space $V$ is the order of the group $P(R)/Q(R)$ where $Q(R)$ is the lattice in $V$ spanned by the elements of $R$ and
\[
P(R) = \{\ell \in V : (\ell,m) \in \ZZ \text{ for all }m \in Q(R) \}.
\]

\begin{pro}
\label{pro:connectedness}
There are three cases for the index of connectedness of the root system $R^l_k$:
\begin{enumerate}
\item if $l=k+2$, the index of connectedness of $R^l_k$ is $2(k+1)$;
\item if $l=k+3$, the index of connectedness of $R^l_k$ is $k+4$;
\item if $l=k+4$, the index of connectedness of $R^l_k$ is $4$.
\end{enumerate}
\end{pro}
\begin{proof}
The proof follows the method of~\cite[Proposition~$25.3$]{Manin86}. Consider the homomorphism
\begin{align*}
\chi\colon P(R^l_k) \rightarrow \QQ/\ZZ \quad \text{for which} \quad \chi\left(a\ell_0 + \sum{b_i\ell_i}\right) = b_1 \mod \ZZ. 
\end{align*}
Writing out the scalar product of $\left(a\ell_0 + \sum{b_i\ell_i}\right) \in P(R^l_k)$ with roots $\ell_1 - \ell_i$ and $\ell_0 - \ell_1 - \ldots - \ell_{k+2}$ the integrality condition implies that,
\begin{align*}
b_1 - b_i \in \ZZ \quad \text{and} \quad ka - b_1 - \cdots - b_{k+2}.
\end{align*}
Furthermore $(k+2)a - \sum\limits^{l}_{i=1}{b_i} = 0$. Thus since $\{\ell_1-\ell_i : 2 \leq i \leq l\}$ and $\ell_0-\ell_1-\ldots-\ell_{k+2}$ jointly generate $N^l_k$, it follows that $\ker(\chi) = N^l_k\cap\omega^\bot$ and
\begin{align*}
ka - (k+2)b_1 \in \ZZ \quad \text{and}  \quad (k+2)a - lb_1 \in \ZZ.
\end{align*}
There are three cases to consider:
\begin{enumerate}
\item $l=k+2$; Then $2a \in \ZZ$, and hence $b_1 \in \frac{1}{2(k+1)}\ZZ$;
\item $l=k+3$; Then $2a-b_1 \in \ZZ$, so $(k+4)a \in \ZZ$. Since $2a-b_1 \in \ZZ$ it follows $b_1 \in \frac{1}{k+4}\ZZ$;
\item $l=k+4$; Then $a-b_1 \in \frac{1}{2}\ZZ$ and hence $b_1 \in \frac{1}{4}\ZZ$.
\end{enumerate}
Thus in each of these three cases $\chi$ is an isomorphism into its image.
\end{proof}

\noindent Consider the index of connectedness of $R^8_3$. In this case
\begin{align*}
3a-5b_1 \in \ZZ \quad \text{and} \quad 5a-8b_1 \in \ZZ.
\end{align*}
However the matrix
\[
\begin{pmatrix}
3 & -5 \\
5 & -8
\end{pmatrix} \in \GL(2,\ZZ)
\]
and thus $b_1 \in \ZZ$ and the index of connectedness of $R^8_3$ is equal to one.

To complete the proof of Theorem~\ref{thm:root_systems} we need to show that all possible roots are classified by Proposition~\ref{pro:big_subsystem}. However, studying the tables in Bourbaki~\cite{Bourbaki}, identify each root system $R^l_k$ using the subsystem found in Proposition~\ref{pro:big_subsystem} and the index of connectedness of $R^l_k$. Make use of the fact that the index of connectedness of a product of root systems is the index of connectedness of its factors. Observe also that all the root vectors in $R^l_k$ have the same length so there are no type $B$ or $C$ factors in the root system $R^k_l$.

In the case $l=k+2$, there are at most two summands, since we have identified orthogonal $A_{k+1}$ and $A_1$ subsystems. Assume there are two factors. One of these is $A_1$ and the other, $R$, contains an $A_{k+1}$ subsystem. Since the index of connectedness of $R$ is equal to $(k+1)$, one larger than its rank, thus $R$ must be of type $A$. Assuming that there is only one summand, there is a contradiction, since the only case with index of connectedness at most four occurs when $k=1$, but the root systems $R^l_1$ are well known.

In the case $l=k+3$ there is at most one summand, of rank $k+3$, and index of connectedness $k+4$. Since $k$ is a positive integer the index of connectedness is always greater than four and thus this root system must be of type $A$.

In the case $l=k+4$ there is at most one summand, of rank $k+4$, and index of connectedness $4$. Thus this root system must be of type $D$.

Since $l \leq k+4$ if $k > 3$ these exhaust all possible cases for general values of $k$. In the case $k=3$ there is a single exceptional case, the root system $R^8_3$ associated to the surface obtained via a section of $\cO(10)$ in the weighted projective space $\PP(1,2,3,5)$. As noted in the discussion following the proof of Proposition~\ref{pro:connectedness}, this root system has index of connectedness equal to one. Therefore $R^8_3$ is of type $E_8$ and the roots can be enumerated similarly to the other cases. The roots of $R^8_3$ are tabulated below, and recall that we are free to permute the $b_i$ and reverse signs to generate roots from the ones listed in this table.

\[
\begin{array}{ccccccccc}
a & b_1 & b_2 & b_3 & b_4 & b_5 & b_6 & b_7 & b_8 \\ \hline
0 & 1 & 0 & 0 & 0 & 0 & 0 & 0 & -1 \\
1 & 1 & 1 & 1 & 1 & 1 & 0 & 0 & 0 \\
2 & 2 & 2 & 1 & 1 & 1 & 1 & 1 & 1 \\
3 & 2 & 2 & 2 & 2 & 2 & 2 & 2 & 1 \\
\end{array}
\]
Permuting all entires $b_i$ and changing signs obtain
\[
2\left( \binom{8}{2} + \binom{8}{3} + \binom{8}{2} + \binom{8}{1} \right) = 240
\]
roots in $R^l_k$. Moreover the Cartan matrix formed from the basis $(\ell_{i+1}-\ell_{i})$ and $(\ell_0 + \ldots + \ell_5)$ is precisely the Cartan matrix of the $E_8$ root system.

%----------------------------------------------------------------------
\section{The Proof of Theorem~\ref{thm:cascade}}
%----------------------------------------------------------------------

The proof of~Theorem~\ref{thm:cascade} is based on the \emph{directed MMP} and has an identical structure to the classification of del~Pezzo surfaces with $\frac{1}{3}(1,1)$ singularities in~\cite{CH16}, although our current task is made considerably simpler by the assumption there is a \emph{single} $\frac{1}{k}(1,1)$ singularity.

\begin{dfn}
\label{dfn:floating_minus_one_curve}
Given a del~Pezzo surface $X$ and rational curve $C \subset X$, then $C$ is a \emph{floating $(-1)$-curve} if $C$ is contained in the smooth locus of $X$ and $C^2 = -1$.
\end{dfn}

We rely heavily on the classification of extremal contractions for surfaces containing a single singular point of the form $\frac{1}{k}(1,1)$. This classification is made in Proposition~\ref{pro:extremal_contractions} and is directly analogous to~\cite[Theorem~31]{CH16}.

\begin{pro}
\label{pro:extremal_contractions}
Given a del~Pezzo surface $X$ with a single singular point of the form $\frac{1}{k}(1,1)$, let $E$ denote the exceptional curve of the minimal resolution $\widehat{X} \rightarrow X$ and let $f\colon X \rightarrow X_1$ be an extremal contraction. Exactly one of the following holds:
\begin{enumerate}
\item the morphism $f$ is the contraction of a floating $(-1)$-curve;
\item the morphism $f$ is the contraction of a $(-1)$-curve in the minimal resolution of $X$ meeting the curve $E$ once. The surface $X_1$ has one singular point of the form $\frac{1}{k-1}(1,1)$ if $k > 1$ and is smooth if $k=2$;
\item the morphism $f$ is a Mori fibre space contraction. In this case $X_1$ is a single point and $X \cong \PP(1,1,k)$.
\end{enumerate}
\end{pro}
\begin{proof}
Fix an integer $k > 1$, let $X$ be a del~Pezzo surface with a single $\frac{1}{k}(1,1)$ singularity and let $\widehat{X} \rightarrow X$ be its minimal resolution with exceptional curve $E$. The surface $\widehat{X}$ is, by construction, a smooth projective surface with big anti-canonical class. Since $\widehat{X}$ has Kodaira dimension $-\infty$, $\widehat{X}$ is a ruled surface, i.e.~$\widehat{X}$ is birational to $\PP^1 \times C$ for some curve $C$. However the only such surface with big anti-canonical class is $\PP^1\times \PP^1$ and hence $\widehat{X}$ is rational.

By the classification of rational surfaces, see for example Beauville~\cite{Beauville}, if $\widehat{X}$ contains no $(-1)$-curves it is isomorphic to the Hirzebruch surface $\FF_k$ (since $\widehat{X}$ contains a negative curve of self-intersection $-k$). Suppose now that $\widehat{X}$ contains a $(-1)$-curve $C$; after contracting all floating $(-1)$-curves and all curves $C$ such that $C.C=-1$, and $C.E = 1$ we have a surface $\widehat{X}_1$. So if $C$ is a rational curve in $\widehat{X}_1$ and $C.C=-1$, then $E.C \geq 2$. Contracting all such curves obtain a surface $\widehat{X}_2$ isomorphic to $\FF_l$ for some $l \in \ZZ_{\geq 0}$, or $\PP^2$. However the last contraction was the blow-up of a point on $\widehat{X}_2$ and this will not meet $E$ in more than one point.
\end{proof}

The list of extremal contractions appearing in Proposition~\ref{pro:extremal_contractions} is much shorter than that appearing in~\cite[Theorem~$31$]{CH16} and consequently the analysis of the directed MMP is much more straightforward. This is due to the presence of exactly one singular point and the simple form of its minimal resolution.

It is also important to ensure that type (ii) divisorial contractions do not introduce more floating $(-1)$-curves. This is analogous to~\cite[Lemma~$33$]{CH16} in our (simpler) context.

%\begin{lem}
%Let $X$ be a del~Pezzo surface with a single $\frac{1}{k}(1,1)$ singularity which contains no floating $(-1)$-curves. Let $f \colon X \rightarrow X_1$ be an extremal contaction of type (ii). The surface $X_1$ contains no floating $(-1)$-curves.
%\end{lem}
%\begin{proof}
%Assume there is a floating $(-1)$-curve $C \subset X_1$. Since $C$ is contained in the smooth locus of $X_1$ it does not meet the exceptional locus of $f$. Therefore $f^{-1}$ is an isomorphism in a neighbourhood of $C$ and $f^{-1}(C)^2 = -1$, a contradiction.
%\end{proof}

\begin{proof}[Proof of Theorem~\ref{thm:cascade}]
Fix an integer $k > 1$, let $X$ be a del~Pezzo surface with a single $\frac{1}{k}(1,1)$ singularity and let $\hat{X} \rightarrow X$ be its minimal resolution with exceptional curve $E$. Assume that there are no floating $(-1)$-curves. Either there is a divisorial contraction (ii) of $X$, or $X$ is the weighted projective space $\PP(1,1,k)$. If $X$ is equal to $\PP(1,1,k)$ we are done. Assuming that $X$ is not isomorphic to $\PP(1,1,k)$ there is a sequence of divisorial contractions and taking the longest possible composition of these $\pi \colon \hat{X} \rightarrow \hat{X}_1$, $\pi(E)^2 = l$ for some $0 \leq l < k $ . If $l > 0$, $\hat{X}_1$ must be isomorphic to $\FF_l$. However blowing up a point in the negative curve of $\FF_l$ introduces a floating $(-1)$-curve, so this cannot occur. If $l = 0$ then $\hat{X} \cong \PP^1\times \PP^1$; it is easily seen that the surface $B^{(k)}_k$ admits such a sequence of contractions.
\end{proof}

%----------------------------------------------------------------------
\section{Surfaces with larger baskets}
%----------------------------------------------------------------------
\label{sec:larger_baskets}
In this section we complete the proof of Theorem~\ref{thm:small_classification}. In particular we classify families of locally $\QQ$-Gorenstein rigid del~Pezzo surfaces with baskets of $R$-singularities of the form
\[
\left\{ m_1 \times \frac{1}{3}(1,1), m_2 \times\frac{1}{5}(1,1), m_3 \times\frac{1}{6}(1,1) \right\},
\]
such that 
\begin{align*}
m_1=0,m_2 > 0,m_3=0 \quad \text{or} \quad m_1 \geq 0,m_2=0,m_3 >0,
\end{align*}
%\begin{rem}
%The remaining cases, those surfaces with $\frac{1}{k}(1,1)$ singularities with $k \in \{2,4\}$ are very closely related to the del~Pezzo surfaces (since $\frac{1}{k}(1,1)$ singularities admit $\QQ$-Gorenstein smoothings in this case) and we adress these cases first.

%\subsection{Very small $k$}

%We briefly review the situation for values of $k < 5$. Clearly the case $k=1$ is classical. The case $k=2$ is also well-understood, since the minimal resolutions of such surfaces are the \emph{weak del~Pezzo surfaces} TO DO: Classification of these.
%The classification in the case $k=3$ was made in~\cite{CH16}, and features only one exceptional surface ($X^{(8)}_3$) with its $E_8$ root system. The case $k=4$ is 
%When $k = 1$ the log del~Pezzo is smooth and we have the $10$ original del~Pezzo surfaces. When $k=2$ or~$4$ the cascades $\{X^l_k,Z_k\}$ admit $\QQ$-Gorenstein smoothings to the del~Pezzo surfaces. 

%\subsection{Surfaces with larger baskets}

% To complete the proof of Theorem~\ref{thm:small_classification} we need to classify those surfaces with multiple R-singularities of the form $\frac{1}{k}(1,1)$ which admit a $\QQ$-Gorenstein toric degeneration. In~\cite{CK17} the authors classify the toric varieties to which such a surface can degenerate and thus, applying Laurent inversion to these cases, we can find obtain models for these surfaces. In fact, following~\cite{CH16}, we can assume that there is at least one rigid singularity $\frac{1}{k}(1,1)$ such that is larger than $3$.

\noindent which admit a  $\QQ$-Gorenstein toric degeneration. The toric varieties to which such a surface can degenerate are classified in~\cite{CK17}; applying Laurent inversion to these cases gives models for these surfaces. The main results of~\cite{CK17} show that either such a surface contains a single $\frac{1}{k}(1,1)$ singularity, for $k \in \{3,5,6\}$, or is one of three exceptional cases. In this section we show that all of these surfaces are hypersurfaces in weighted projective spaces. In particular, consider polygons $1.13$ and $1.14$ from~\cite{CK17}. While we use Laurent inversion here we could also use the Ehrhart series of the dual polygons to those appearing in~\cite{CK17} to guess the hypersurface model.
\medskip

Polygon $1.13$ is given by $\conv{ \big\{ (-1,1) , (1,1) , (5,-1), (-5,-1) \big\} }$. After mutating the $T$-singularities to the top edge obtain the polygon $P = \conv{\big\{ (-6,-1) , (0,1) , (6,-1) \big\}}$ (up to $GL(N)$-equivalence). Use the following scaffolding of $P$ consisting of two struts:

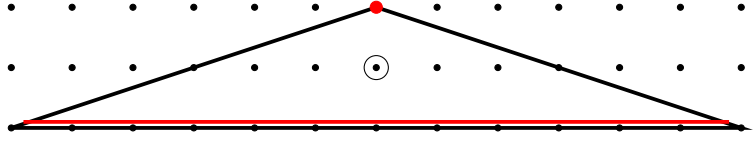
\begin{center}
\begin{tikzpicture}[scale=0.8, transform shape]
\begin{scope}
\clip (-6.3,-1.3) rectangle (6.3cm,1.3cm); % Clips the picture...
\draw[line width=0.5mm,black] (-6,-1) -- (6,-1) -- (0,1) -- (-6,-1); % Puts the shaded rectangle
\draw[line width=0.5mm, red] (-5.8,-0.9) -- (5.8,-0.9); 
\foreach \x in {-7,-6,...,7}{                           % Two indices running over each
    \foreach \y in {-7,-6,...,7}{                       % node on the grid we have drawn 
    \node[draw,shape = circle,inner sep=1pt,fill] at (\x,\y) {}; % Places a dot at those points
    }
}
 \node[draw,shape = circle,inner sep=4pt] at (0,0) {}; % Places a dot at those points
  \node[draw,shape = circle,color=red,inner sep=2pt,fill] at (0,1) {}; 
\end{scope}
\end{tikzpicture}
\end{center}

\begin{enumerate}
\item the single point $\big\{ (0,1) \big\}$;
\item the segment \big[  (-6,-1) , (6,-1) \big].
\end{enumerate}
\noindent By Laurent inversion obtain the weight matrix
\[ \mathcal{M} = \begin{pmatrix} 1 & 6 & 6 & 1 \end{pmatrix}. \]
Therefore $X_{P}$ is given by the a general section of $\cO(12)$ in $\PP(1,1,6,6)$. By Theorem~\ref{thm:quasismooth_codim_1} $X_P$ is quasismooth and also $X_{P}$ inherits two $\frac{1}{6}(1,1)$ from the ambient weighted projective space.

Via a different scaffold it is possible to obtain a different model. Mutate our original representative of polygon 1.13, namely $P_{1.13} := \conv{\big\{ (-1,1) , (1,1) , (5,-1), (-5,-1) \big\}}$, to the representative by $\conv{\big\{ (-3,1) , (3,1) , (3,-1), (-3,-1) \big\}}$. Scaffold $P_{1.13}$ via a single strut as shown below:

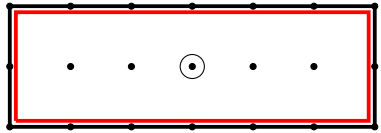
\begin{center}
\begin{tikzpicture}[scale=0.8, transform shape]
\begin{scope}
\clip (-3.3,-1.3) rectangle (3.3cm,1.3cm); % Clips the picture...
\draw[line width=0.5mm,black] (-3,-1) -- (3,-1) -- (3,1) -- (-3,1) -- (-3,-1); % Puts the shaded rectangle
\draw[line width=0.5mm, red] (-2.9,-0.9) -- (2.9,-0.9) -- (2.9,0.9) -- (-2.9,0.9) -- (-2.9,-0.9); 
\foreach \x in {-7,-6,...,7}{                           % Two indices running over each
    \foreach \y in {-7,-6,...,7}{                       % node on the grid we have drawn 
    \node[draw,shape = circle,inner sep=1pt,fill] at (\x,\y) {}; % Places a dot at those points
    }
}
 \node[draw,shape = circle,inner sep=4pt] at (0,0) {}; % Places a dot at those points
\end{scope}
\end{tikzpicture}
\end{center}

\noindent Laurent inversion gives the weight matrix
\[
\mathcal{M} =
\begin{pmatrix}
1 & 1 & 1 & 3 & 3
\end{pmatrix},
\]
and the corresponding toric variety is the complete intersection of the vanishing of two general sections of $\cO(2)$ and $\cO(6)$ in $\PP(1,1,1,3,3)$. It is routine to check that this has the appropriate singularities.

In fact the two models
\begin{align*}
\PP(1,1,6,6), \cO(12), \\
\PP(1,1,1,3,3), \cO(2)\oplus\cO(6),
\end{align*}
are isomorphic. This can be seen by observing that (possibly after a change of co-ordinates) the vanishing locus of a general section of $\cO(2)$ on $\PP(1,1,1,3,3)$ is isomorphic to the image of the degree $2$ Veronese embedding $\PP(1,1,6,6) \hookrightarrow \PP(1,1,1,3,3)$ defined by sending
\[
(x_1,x_2,y_1,y_2) \mapsto (x^2_1,x_1x_2,x^2_2,y_1,y_2).
\]

\medskip

In fact the hypersurface model of these surfaces generalises to a construction of a del~Pezzo surface with a pair of $R$-singularities $\frac{1}{k_1}(1,1)$,~$\frac{1}{k_2}(1,1)$ for any pair of positive integers $k_1$,~$k_2$. Consider the polygon $P$ with vertices $(0,1),(-k_1,-1),(k_2,-1)$. Scaffold using the struts as illustrated:
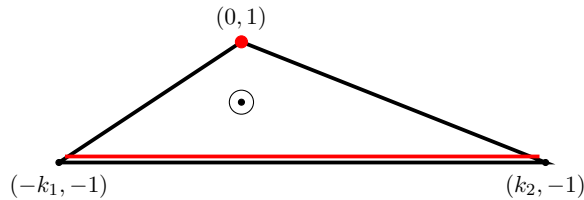
\begin{center}
\begin{tikzpicture}[scale=0.8, transform shape]
\begin{scope}
\clip (-4,-1.6) rectangle (6cm,1.6cm); % Clips the picture...
\draw[line width=0.5mm,black] (-3,-1) -- (0,1) -- (5,-1) -- (-3,-1); % Puts the shaded rectangle
\draw[line width=0.5mm, red] (-2.9,-0.9) -- (4.9,-0.9); 
\node[draw,shape = circle,inner sep=1pt,fill] at (0,0) {}; 
\node[draw,shape = circle,inner sep=1pt,fill,label = below:{$(k_{2},-1)$}] at (5,-1) {}; 
\node[draw,shape = circle,inner sep=1pt,fill,label = below:{$(-k_{1},-1)$}] at (-3,-1) {}; 
\node[draw,shape = circle,inner sep=1pt,fill,label = above:{$(0,1)$}] at (0,1) {}; 
 \node[draw,shape = circle,inner sep=4pt] at (0,0) {}; % Places a dot at those points
   \node[draw,shape = circle,color=red,inner sep=2pt,fill] at (0,1) {}; 
\end{scope}
\end{tikzpicture}
\end{center}

%In the language of toric geometry, 
This polygon has two R-cones representing $\frac{1}{k_{1}}(1,1)$ and $\frac{1}{k_{2}}(1,1)$ cyclic quotient singularities. Laurent inversion gives us the weight matrix:
\[ \mathcal{M} = \begin{pmatrix} 1 & k_{1} & k_{2} & 1 \end{pmatrix}. \]

\noindent Thus the toric variety $X_P$ is a subvariety of $\PP_{(x_{1}:x_{2}:y_{1}:y_{2})}(1,1,k_1,k_2)$ cut out by the equation
\[
y_1y_2 - x^{k_1}_1x^{k_2}_2.
\]
Consider the del~Pezzo surface given by the vanishing of a general section of $\mathcal{O}(k_{1}+k_{2})$ on $\PP(1,1,k_1,k_2)$. By Theorem~\ref{thm:quasismooth_codim_1} the surface is quasismooth and the only singularities are inherited from the ambient space. Assume $k_{1} \neq k_{2}$ and without loss of generality $k_{1} < k_{2}$ so that $k_{2}=nk_{1}+r$. If $r=0$, then a general section of $\mathcal{O}(k_{1}+k_{2})$ is given by 
\[ f = \sum\limits_{i=0}^{n-1} f_{(1-i)k_{1}+k_{2}}(x_{0},x_{1}) y^{i} + yz + y^{n}, \]
where $x_{0},x_{1},y,z$ are coordinates on $\mathbb{P}(1,1,k_{1},k_{2})$. Then the surface intersects the orbifold locus at the points $[0:0:0:1]$ and $[0:0:1:-1]$ giving cyclic quotient singularities $\frac{1}{k_{1}}(1,1)$ and $\frac{1}{k_{2}}(1,1)$ respectively. If $r \neq 0$, then a general section is given by 
\[ f = \sum\limits_{i=0}^{n} f_{(1-i)k_{1}+k_{2}}(x_{0},x_{1}) y^{i} + yz. \]
The zero locus of $f$ intersects the orbifold locus at $[0:0:0:1]$ and $[0:0:1:0]$ giving cyclic quotient singularities $\frac{1}{k_{1}}(1,1)$ and $\frac{1}{k_{2}}(1,1)$ on the del~Pezzo surface. The case of $k_{1}=k_{2}$ is treated similarly.

\begin{cor}
There exists a del~Pezzo surface admitting a toric degeneration with exactly two $R$-singularities $\frac{1}{k_{1}}(1,1)$ and $\frac{1}{k_{2}}(1,1)$ given by the vanishing of a general section of $\cO(k_{1}+k_{2})$ on $\PP(1,1,k_{1},k_{2})$. Considering the local models near the smoothable singularities of the respective toric varieties it is easily verifiable that this deformation is $\QQ$-Gorenstein.
\end{cor}

\begin{rem}
\label{rem:pairs_of_singularities}
Of course, the surfaces appearing in Theorem~\ref{thm:small_classification} with more than one $R$-singularity admit models as sections of $\cO(k_1+k_2)$ in $\PP(1,1,k_1,k_2)$. There are four cases with $R$-singularites $\frac{1}{k}(1,1)$ with $k < 7$, these are the del~Pezzo surfaces
\begin{enumerate}
\item $X_8 \subset \PP(1,1,3,5)$ defined by a general section of $\cO(8)$;
\item $X_9 \subset \PP(1,1,3,6)$ defined by a general section of $\cO(9)$;
\item $X_{10} \subset \PP(1,1,5,5)$ defined by a general section of $\cO(10)$;
\item $X_{11} \subset \PP(1,1,5,6)$ defined by a general section of $\cO(11)$.
\end{enumerate}
Of these $X_9$ and $X_{10}$ are needed to complete the proof of Theorem~\ref{thm:small_classification}.
\end{rem}

%\subsection{Surfaces without a $\QQ$-Gorenstein toric degeneration}
%\label{sec:no_degeneration}
%TO DO: Blow up our four surfaces in a non singular point and prove that there is no toric degeneration.

%----------------------------------------------------------------------
\section{Mirror Symmetry}
\label{sec:mirror_symmetry}
%----------------------------------------------------------------------

\subsection{Mutation classes of polygons}
\label{sec:conjecture_A}

It is vital to understand the notion of mutations introduced by Akhtar--Coates--Galkin--Kasprzyk~\cite{ACGK}. The constructions used throughout this article produce a smoothing $X$ of the toric variety $X_P$ associated to a Fano polygon $P$ embedded in a toric variety of higher dimension. Mirror Symmetry can be studied in~\cite{A+,CCGGK}. A general conjecture, inspired by Mirror Symmetry, is made in~\cite{A+} to describe the set of toric varieties to which $X$ degenerates:

\begin{conjecture}[{\cite[Conjecture~A]{A+}}]
There is a canonical bijection between the set of mutation equivalence classes of Fano polygons and deformation families of $\QQ$-Gorenstein locally rigid del~Pezzo surfaces with cyclic quotient singularities which admit a toric degeneration.
\end{conjecture}

Since $\QQ$-Gorenstein deformations of surfaces are unobstructed (see~\cite{A+}) to verify Conjecture~A for Fano polygons with a specified basket of $R$-singularities it is sufficient to identify the mutation classes of Fano polygons with these singularities, and verify that their respective $\QQ$-Gorenstein deformations are never isomorphic.

\begin{pro}
Conjecture~A holds for del~Pezzo surfaces with the baskets of singularities which appear in statement of Theorem~\ref{thm:small_classification}.
\end{pro}
\begin{proof}
Akhtar--Kasprzyk~\cite{AK14} observe that the topological Euler characteristic of the smooth locus of a general $\QQ$-Gorenstein deformation of a toric Fano surface $X_P$ can be read from the Fano polygon $P$ and this is precisely the notion of \emph{singularity content}. Singularity content distinguishes every mutation class of polygons classified in~\cite{CK17} except those describing toric degenerations of the surfaces $X^{(k)}_k$ and $B^{(k)}_k$. Thus it is sufficient to show that these two surfaces are not deformation equivalent. To do this use a finer topological invariant considered in~\cite{KNP15}: the fundamental group of the complement of a general anti-canonical divisor. This can be computed from the Fano polygon $P \subset N_\QQ$ of a $\QQ$-Gorenstein toric degeneration by taking the quotient $G$ of $M$ by a lattice generated by all possible weight vectors of mutations of $P$. It is easy to see that $G$ is trivial in the case $X^{(k)}_k$, but $G \cong \ZZ_2$ in the case $B^{(k)}_k$.
\end{proof}

\subsection{Mirror Symmetry via Quivers}
\label{sec:clusters}

Mirror Symmetry for Fano varieties conjectures a correspondence between a given Fano variety together with a choice of anti-canonical divisor $(X,D)$ and a certain \emph{Landau--Ginzburg model} $(U,W)$. For us, a Landau--Ginzburg model is a pair $(U,W)$, where $U$ is a K\"{a}hler manifold equipped with a holomorphic function $W$. Following the results and constructions appearing in~\cite{A+,GHK2,GHK1,KNP15} there is a well-understood mirror model for each of the surfaces $X^{(l)}_k$. In this section we recall this construction and tabulate the mirror-dual models for each of the surfaces $X^{(l)}_k$. We omit proofs of the statements in this section, referring the reader to the papers~\cite{A+,GHK2,GHK1,KNP15} which deal with various aspects of this construction.

Fix a pair $(k,l)$ so that $X := X^{(l)}_k$ is a del~Pezzo surface and an element $D \in |-K_X|$. Assume throughout this section that $k=3$ or $k > 4$ to reduce the number of cases that need to be considered. The construction of $U$ follows that given in~\cite{GHK2,GHK1} for general log Calabi--Yau surfaces with maximal boundary. The algorithm to construct $U$ is most easily seen via a toric degeneration $X_0$ of $X$.

\begin{algorithm}
\label{alg:mirror_log_CY}
Fix the degeneration of $X$ to the toric variety $X_P$ where $P = P^l_k$ is specified in~\S\ref{sec:low_codim_models}. We construct the mirror-dual log Calabi--Yau $U$ in three stages:
\begin{enumerate}
\item Let $Y_0$ be the toric variety associated to the normal fan $\Sigma_P$ of $P$.
\item For each ray $\rho \in \Sigma_P(1)$ choose $a_\rho$ points $\{p_{i,\rho} : i \in [a_\rho]\}$ on the corresponding divisor of $Y_0$, where $0 \leq a_\rho \leq m_\rho$, and $m_\rho$ is the \emph{singularity content} of the torus fixed point of $X_P$ determined by $\rho$. 
\item Blow-up all the points in $\bigcup_{\rho \in \Sigma_P(1)} \{p_{i,\rho} : i \in [a_\rho]\}$ and define $U$ to be the complement of the strict transform of the toric boundary of $Y_0$.
\end{enumerate}
\end{algorithm}

There is a choice made in Algorithm~\ref{alg:mirror_log_CY} in the number of points $p_{i,\rho}$ on various divisors. This corresponds precisely to the choice of the number of irreducible components of the anti-canonical divisor $D$. 

Gross--Hacking--Keel~\cite{GHK2} describe how to attach a quiver (and hence a cluster algebra) to the log Calabi--Yau $U$ together with a \emph{toric model}. An equivalent quiver $\cQ_P$ constructed from the Fano polygon $P$ (via Algorithm~\ref{alg:mirror_log_CY}) is described in~\cite{KNP15}. In~\cite{GHK2} it is observed that, up to the taking the complement of a codimension two subvariety, Mirror Symmetry in this context is precisely the duality between the $\cX$ and $\cA$ type cluster varieties appearing in the work of Fock--Goncharov~\cite{FG09}. 

We now recall the construction appearing in~\cite{KNP15} of $\cQ_P$ from the Fano polygon $P$ and tabulate a choice of quiver for each of the del~Pezzo surfaces with a single $\frac{1}{k}(1,1)$ singularity. Let $P$ be a Fano polygon with singularity content given by the pair $(n,\mathcal{B})$. The quiver $\cQ_P$ has $n$ vertices, and each vertex $v_{i}$ of $\cQ_P$ corresponds to a $T$-singularity of $P$ which lies on an edge $E$. Let $\omega_{i}$ be the inward pointing normal to $E$. The number of arrows in $\cQ_P$ from $v_{i}$ to $v_{j}$ is given by
\[
\max \left\{ \omega_i \wedge \omega_j, 0 \right\},
\]
where we have fixed an orientation of the lattice $M$ containing the normal directions to the edges of $P$.
%\[
%\omega_{i} = \begin{pmatrix} \omega_{i}^{1} \\ \omega_{i}^{2} \end{pmatrix}
%\]
%be the inward pointing normal to $E$. The number of arrows in $\cQ_P$ from $v_{i}$ to $v_{j}$ is given by
%\[
%\max \bigg\{ \det \begin{pmatrix} \omega_{i}^{1} & \omega_{j}^{1} \\ \omega_{i}^{2} & %\omega_{j}^{2} \end{pmatrix} , 0 \bigg\}.
%\]

In fact it is often useful to use a smaller quiver $\cQ'_P$, the subquiver of $\cQ_P$ obtained by forgetting a single node of $\cQ_P$ corresponding to each Gorenstein singularity (In particular remove all nodes corresponding to smooth cones). For example, if $X_P \cong \PP^2$, $\cQ_P$ is a cycle with three arrows between each node, whereas $\cQ'_P$ is empty. We tabulate those quivers $\cQ'_P$ obtained from the surfaces $X^{(l)}_k$. Note that (unlike $\cQ_P$) the number of nodes of $\cQ'_P$ depends on $P$ and not only its mutation equivalence class. Also note that each of the polygons $P$ used to populate the table is related to $P^{(l)}_k$ by polygon mutation (but are not equal in general).

\begin{center}
{\renewcommand{\arraystretch}{2}
\begin{tabular}{ c|c|c }
$l$ & $\cQ'_P$ & $\#$ components of $D$ \\  \hline
$0$ & $\varnothing$ & $3$  \\
$1$ & $\varnothing$ & $4$  \\
$2$ & $\varnothing$ & $5$  \\
$2 \leq l \leq k+1$ & $A^{l-2}_1$ & $5$ \\
$k+2$ & $\xymatrix@=0.5pc{ \bullet & \cdots k \cdots & \bullet\\ & \bullet \ar[ul] \ar[u] \ar[ur] & }$ & 4 \\
$k+3$ & $\xymatrix@=0.5pc{ \bullet & \cdots (k+2) \cdots & \bullet\\ & \bullet \ar[ul] \ar[u] \ar[ur] & }$ & 3\\
$k+4$ &  $\xymatrix@=0.5pc{ \bullet \ar[drr] & \cdots (k+1) \cdots \ar[dr] & \bullet \ar[d]\\ \bullet \bullet \ar[u] \ar[ur] \ar[urr] &  & \bullet \bullet \ar[ll] }$ & 2
\end{tabular}}
\end{center}
\medskip
The log Calabi--Yau variety  $U$ mirror to $(X^{(l)}_k,D)$, with $D$ as indicated in the above table, is the $\cA$-type cluster variety associated to $\cQ'_P$. The choice of $D$ determines a holomorphic function $W$ on $U$, that is, an element of the upper cluster algebra associated to $\cQ'_P$. Fix such a function by observing that, by construction, each torus chart in $U$ is associated with a Fano polygon $P$ and requiring that the Newton polyhedron of $W$ restricted to this chart is equal to this polygon. This definition precisely coincides with the notion of \emph{maximally mutable Laurent polynomial}~\cite{A+,Kasprzyk--Tveiten}. In fact the choice of $\cQ_P$ or $\cQ'_P$ does not matter: the possible functions $W$ are the same.

As explained in~\cite{A+,Kasprzyk--Tveiten}, this function is not unique. In a way made precise in~\cite{A+}, $W$ depends on a number of parameters determined by the residual singularities of $X_P$ (in the present case the single $\frac{1}{k}(1,1)$ singularity). The parameters which appear are related to the orbifold Quantum cohomology of $X^{(l)}_k$ and were studied by Oneto--Petracci~\cite{OP15} when $k=3$.

Note that, as well as its intrinsic interest, a cluster algebra description of the surfaces $X^{(l)}_k$ provides deep geometric insights. Indeed, in~\cite{GHKK} Gross--Hacking--Keel--Kontsevich study canonical bases of functions for such varieties via theta functions, which appeared in~\cite{GHK1}. In~\cite{GHK3} Gross--Hacking--Keel prove a Torelli type theorem for log Calabi--Yau varieties, meaning the families of surfaces considered should be accessible via a certain period map.

%----------------------------------------------------------------------
\subsection*{Acknowledgements}
%----------------------------------------------------------------------

We thank Alexander Kasprzyk, Alessio Corti and Andrea Petracci for many useful conversations. DC particularly thanks his doctoral advisor, Alexander Kasprzyk, for his guidance. DC would also like to thank Liana Heuberger for her help in his understanding of~\cite{CH16}, as well as Miles Reid and Gavin Brown for useful discussions. TP was supported by an EPSRC Doctoral Prize Fellowship and Tom Coates' ERC Grant 682603. This work was partially supported by Kasprzyk's EPSRC Fellowship EP/N022513/1.

% TO DO: I think feel free to add more names to this list, those are all the people I've had advice from on this project.

%-------------------------------------------------------------------------------

%-------------------------------------------------------------------------------
\end{document}